\documentclass[leqno,11pt]{amsart}
\usepackage{amssymb, amsmath,amsmath,latexsym,amssymb,amsfonts,amsbsy, amsthm,esint }
\usepackage{color}
\usepackage{float}
\usepackage{graphicx}
\usepackage{tikz}
\usetikzlibrary{arrows, calc, decorations.markings,intersections}
\usepackage{caption}
\usepackage{subcaption}
\usepackage{hyperref}

\makeatletter
\DeclareFontFamily{U}{tipa}{}
\DeclareFontShape{U}{tipa}{m}{n}{<->tipa10}{}
\newcommand{\arc@char}{{\usefont{U}{tipa}{m}{n}\symbol{62}}}%

\newcommand{\arc}[1]{\mathpalette\arc@arc{#1}}

\newcommand{\arc@arc}[2]{%
  \sbox0{$\m@th#1#2$}%
  \vbox{
    \hbox{\resizebox{\wd0}{\height}{\arc@char}}
    \nointerlineskip
    \box0
  }%
}
\makeatother

\let\pa=\partial
\let\al=\alpha

\let\g=\gamma
\let\d=\delta

\let\lam=\lambda

\let\s=\sigma
\let\f=\frac

\let\p=\psi

\let\G= \Gamma
\let\D=\Delta

\let\T=\Theta
\let\Om=\Omega

\let\e=\varepsilon
\let\pa=\partial

\let\ri=\rightarrow
\let\na=\nabla

\def\di{\mathrm{div}\,}
\def\curl{\mathrm{curl}\,}

\newcommand{\beq}{\begin{equation}}
\newcommand{\eeq}{\end{equation}}
\newcommand{\beqo}{\begin{equation*}}
\newcommand{\eeqo}{\end{equation*}}
\newcommand{\ben}{\begin{eqnarray}}
\newcommand{\een}{\end{eqnarray}}
\newcommand{\beno}{\begin{eqnarray*}}
\newcommand{\eeno}{\end{eqnarray*}}

\newtheorem{theorem}{Theorem}[section]

\newtheorem{lemma}[theorem]{Lemma}
\newtheorem{proposition}[theorem]{Proposition}
\newtheorem{corol}[theorem]{Corollary}

\theoremstyle{remark}

\newtheorem{case}{Case}
\newtheorem{rmk}{Remark}[section]

\newcommand{\BR}{\mathbb{R}}

\newcommand{\tgg}{\tilde{\mathcal{G}}^1}
\newcommand{\dist}{\mathrm{dist}}
\newcommand{\mte}{\mathcal{T}_\e}

\newenvironment{problem}[2][Problem]{\begin{trivlist}
\item[\hskip \labelsep {\bfseries #1}\hskip \labelsep {\bfseries #2.}]}{\end{trivlist}}

\begin{document}

\title[Two Dimensional Liquid Crystal Droplet]{Two Dimensional Liquid Crystal Droplet Problem with Tangential Boundary Condition}

\author{Zhiyuan Geng}
\address{Basque Center for Applied Mathematics, Alameda de Mazarredo 14
48009 Bilbao, Bizkaia, Spain}
\email{zgeng@bcamath.org}

\author{Fanghua Lin}
\address{Courant Institute, New York University, 251 Mercer Street, New York, NY 10012, USA}
\email{linf@cims.nyu.edu}

\begin{abstract}
This paper studies a shape optimization problem which reduces to a nonlocal free boundary problem involving perimeter. It is motivated by a study of liquid crystal droplets with a tangential anchoring boundary condition and a volume constraint. We establish in 2D the existence of an optimal shape that has two cusps on the boundary. We also prove the boundary of the droplet is a chord-arc curve with its normal vector field in the VMO space. In fact, the boundary curves of such droplets belong to the so-called Weil-Petersson class. In addition, the asymptotic behavior of the optimal shape when the volume becomes extremely large or small is also studied.
\end{abstract}

\maketitle

\section{Introduction}
\subsection{Background}

Liquid crystal droplets are of great interest from both the theory and applications. They are important in the studies of topological defects in the bulk or on the surface of liquid crystals; and they are useful in understandings of anisotropic surface energies and variety anchoring conditions. Determining the shape of the droplets and the associated equilibrium configurations of the liquid crystals leads to a shape optimization problem that, in some cases, it becomes a nonlocal free boundary.


In fact, we are particularly interested in the elongated droplets known as \emph{tactoids}, which usually possess a characteristic eye shape.
After a quick examination, one finds the boundary anchoring condition for the molecular orientation to achieve such a desired shape needs to be a tangential anchoring  , i.e. the director is orthogonal to the normal of the droplet boundary.


Mathematically, the most commonly used continuum theory to describe nematic liquid crystals is the Oseen-Frank theory, where the local state of the liquid crystal is described by a $\mathbb{S}^1$- or $\mathbb{S}^2$- valued vector $n$ that represents the mean local orientation of molecule's optical axis. Let $\Om$ be the region occupied by a nematic liquid crystal droplet, the Oseen-Frank bulk energy associated with the director field is the functional
\beq\label{oseenfrank}
E_{OF}(n,\Om)=\int_{\Om}w(n.\na n)\,dx,
\eeq
where
\begin{align}
\label{ofdensity} w(n,\na n)=&k_1(\di n)^2+k_2(n\cdot\curl n)^2+k_3|n\times \curl n|^2\\
\nonumber &+(k_2+k_4)\left( \mathrm{tr}(\na n)^2-(\di n)^2 \right).
\end{align}
We shall consider the one-constant approximation, i.e.,  $k_1=k_2=k_3=1$ and $k_4=0$, the energy functional \eqref{oseenfrank} reduces to
\beq\label{harmonicmap}
E_{OF}(n,\Om)=\int_{\Om}|\na n|^2\,dx,
\eeq
which is the energy functional for harmonic maps.

 Liquid crystal droplets are often either dispersed in an polymeric medium or surrounded by another fluid such as water, there is an interfacial energy which will play an essential role in determining the optimal shapes.
Following \cite{Friedel} and \cite{oseen}, the surface energy may be written as
\beq\label{surface}
E_{s}(\Om,n)=\int_{\pa\Om} f(n\cdot \nu)\,dA
\eeq
where $\nu$ is the outer normal on $\pa\Om$ and for simplicity, $f$ is assumed to have the form (see \cite{ericksen})
\beq\label{form of f}
f(\theta)=\mu(1+\lambda\theta^2),
\eeq
for some $\mu>0$ and $-1<\lambda<\infty$.
Thus the total energy for a liquid crystal droplet configuration is given by:
\beqo
E(\Om,n)=E_{OF}(\Om,n)+E_s(\Om,n)
\eeqo

As both the shape of $\Om$ and the director $n$ are varying, determining the stable configuration leads to the following free boundary problem:
\begin{problem}{A}
Find a pair $(\Om,n)$, that minimizes the functional
\beq\label{energy:droplet}
E(\Om,n)=\int_{\Om}w(n,\na n)\,dx+\int_{\pa\Om}f(n\cdot\nu) dA.
\eeq
subject to the constraint $\mathrm{vol}(\Om)=V_0$.
\end{problem}
Here $\mathrm{vol}$ denotes Lebesgue measure and $V_0$ is a positive constant.

Problem A draws great attention from both physicists and mathematicians. There are many research works on Problem A with physical experiments, numerical simulations and formal analysis, see for example \cite{da,ksl,ps,ps0,lc,rb,vov,kkpyv,msw}.  On the other hand, rigorous theoretical treatment of this problem is more challenging because of the difficulty of determining the shape and the director at the same time. One way to overcome such difficulty is to assume the droplet have a simple geometry, such as a disk, an ellipse or a intersection region of two disks, see e.g. \cite{kbs,vov,williams}. In these works, the shape of the droplet is either fixed, or determined by only one or two parameters (such as the eccentricity of an ellipse). And the minimization often involves finding the best shape parameter and the director field under various boundary conditions and different Oseen-Frank elastic constants. Another way is to presume the configuration of director field (such as a constant vector field), and then find the best shape that minimize the surface energy alone, subject to the fixed volume constraint, see e.g. \cite{rb,virga}. These two methods are useful to partially justify the phenomena observed in experiments but are not satisfactory from a mathematical point of view.

A more rigorous study of Problem A was conducted by the second author and Poon in \cite{LinPoon}. Under the key assumption that all admissible domains are convex, they establish the existence and partial regularity of Problem A (see \cite[Theorem 2.4]{LinPoon}). The convexity assumption on $\Om$ the shape of droplets, on one hand, makes the problem more accessible mathematically; and on the other hand, it does match many experimental observed liquid crystals droplets which are of shapes of ellipsoids (balls) and cigars. In this connection, they also studied the cases when the surface energy favors the normal boundary anchoring condition or the tangential boundary anchoring condition. When $\lam>0$ and $\mu\ri \infty$, we get the following minimization problem:
\begin{problem}{B}
(Problem B in \cite{LinPoon})
Find a pair $(\Om,n)$ that minimizes
\beqo
\int_{\Om}w(n,\na n)\,dx+\mu \mathrm{Area}(\pa\Om)
\eeqo
and such that (i) $\mathrm{vol}(\Om)=V_0$ and (ii) $n\cdot \nu=0$ on $\pa\Om$.
\end{problem}

When $-1<\lam<0$ and $\mu\ri \infty$, one gets
\begin{problem}{C}
(Problem C in \cite{LinPoon})
Find a pair $(\Om,n)$ that minimizes
\beqo
\int_{\Om}w(n,\na n)\,dx+\mu \mathrm{Area}(\pa\Om)
\eeqo
and such that (i) $\mathrm{vol}(\Om)=V_0$ and (ii) $n\cdot \nu=1$ on $\pa\Om$.
\end{problem}

It is proven in \cite{LinPoon} that there are minimizers among convex domains $\Om$ for both Problem B and Problem C. Moreover, the only solution to Problem C (up to a Euclidean motion) is $(B_R,\f{x}{|x|})$, such that $|B_R|=V_0$.

Li $\&$ Wang recently extends the previous result in which they Replace the convexity assumption by a notion of M-uniform domains,
see \cite{qli}.
It is worth to point out that the Problems A, B and C thus presented were all derived from a phenomenological theory, see \cite{oseen} and \cite{ericksen}. In a recent work \cite{LinWang}, it is shown that one can rigorously establish these model problems from a general theories of Ericksen (for liquid crystals with variable degree of orientations \cite{ericksen2}) or from the de Gennes-Landau model of liquid crystals \cite{deGennes} in suitable physical regimes.

From our experience, One likely can establish a general existence and partial regularity theory for Problems A, B and C without the convexity assumption on the shape of $\Om$. However, one also expects such a theory will not be able to tell certain particular shapes and configurations (that are observed in experiments and numerical simulations) are minimizers. In particular, one likely will not be able to deduce that \emph{tactoids}, balls, cigars and apples shaped droplets are minimizers. The latter are in fact commonly observed in experiments and of interest to many researchers.
In this article, we will concentrated on the two dimensional case of Problem B, where the tangential anchoring boundary condition and the fixed volume constraint is presumed. The minimizer is expected to have a spindle shape, which is known as \emph{tactoids}, and a \emph{bipolar} director field. Here the bipolar direct field refers to an axially symmetric configuration with tangential anchoring boundary condition, such that two boojums are located at opposite ends of the axis. If one investigate thin liquid crystals samples in experiments, the region of nematic liquid crystals will form a planar domain (tactoid) whose boundary consists of two curves that meet at two singular points and form angles or cusps. For more experimental evidences and numerical simulations of tactoids with such bipolar director configurations, the readers are referred to \cite{da,ksl,ps,ps0,rb,vov} for more details. These works also manifest the significance of tactoids as an object of study.

There are several works that focus on the rigorous mathematical analysis of tactoids with tangential anchoring of the director on the surface. Shen \textit{et al.} \cite{shen} discussed such bipolar configurations of droplet in the fixed spherical domain case as well as the free boundary case. For the latter, they introduce a relaxed energy to establish the existence of critical points and some stability results. Recently, a model problem based on highly disparate elastic constants is proposed by  Golovaty, Novack, Sternberg and Venkatraman in \cite{gnsv} to understand corners and cusps that form on the nematic-isotropic interface. They prove some $\G$-convergence results (when some elastic constant $\e$ goes to $0$) and study the role played by the boundary tangency requirement and the elastic anisotropy on the formation of interfacial singularities.

In this work, we investigate the planar tactoids by solving Problem B. What distinguishes our work from the previous work of Lin $\&$ Poon is that we drop the convexity assumption on the domain $\Om$. Instead, we only assume a symmetry assumption with respect to $x$-axis for the purpose of convenience. We first prove some geometric properties of the free boundary. The main property is that away from two cusps, the boundary curve is a vanishing chord-arc curve and the boundary normal vector $\nu$ is in $\mathrm{VMO}$. Furthermore, we notice that our curve $\G$ has many similar properties with the so-called Weil-Petersson curve (see Section \ref{weil-p}). As a consequence,  the arc-length parameterization of the curve is in the Sobolev space $H^{3/2}$.  Then using these properties, we demonstrate the existence of global minimizer with two cusps on the boundary, which verifies the shape of tactoids. We also study the asymptotic shape of the nematic drop when the volume tends to be very large or very small. Note that due to a very strong non-local character of this problem, currently we are not able to show that $\nu$ is continuous on the boundary. We hope to prove higher regularity results in the future.

\subsection{Mathematical Formulation}
Now we give the precise formulation of the model problem. Note that what we have in mind is the tactoid that forms two cusps on the boundary. Set $\Om\subset \mathbb{R}^2$ as the simply-connected region which is a domain enclosed by a Jordan curve with finite length.  We denote by $n\in \mathbb{S}^1$ the unit vector that represents the director of liquid crystal. The Oseen-Frank bulk energy is given by \eqref{harmonicmap}. Then the variational problem is
\begin{problem}{D}
(2D case of Problem B) Find a pair $\{\Om, n\}$ that minimizes
\beqo
\int_\Om |\na n|^2dx+ \mathrm{Per}(\pa \Om)
\eeqo
such that $\mathrm{vol}(\Om)=V_0 $ and $n\cdot \nu=0 $ on $\pa\Om$. Here $\mathrm{Per}$ means the perimeter.
\end{problem}
\par
Here we want to point out that this formulation already implies that the boundary of the minimizer $\Om$ cannot be smooth everywhere. We can briefly explain it in this way: if $\pa\Om$ is a closed smooth curve and the boundary tangential vector is continuous, then the topological degree of tangential vector is at least one and therefore there is no finite Dirichlet energy extension of $n|_{\pa\Om}$ inside the 2D domain $\Om$. Now we refine this problem by adding more constraints and then introduce the final version of the problem that we will study.

First we assume $\Om$ is symmetric with respect to $x$-axis. And therefore we only consider half of the domain located in the upper-half plane. Let $\Gamma$ be a rectifiable curve that satisfies following conditions:
\begin{itemize}
\setlength\itemsep{0.5em}
 \item[(i)] $\G=\{(x(t),y(t)): x(t),\,y(t)\in AC([0,l(\Gamma)])\}$, where $l(\Gamma)$ is the length of $\Gamma$.
 \item[(ii)] $ (x(0),y(0))=(0,-a),\, (x(1),y(1))=(0,a) \text{ for some }a>0$.
 \item[(iii)] $\mathcal{H}^{1}(\G\cap \{(x,0):x\in\mathbb{R}\})=0$.
 \item[(iv)] $\dot{x}(t)\geq 0 ,\; y(t)\geq 0,\;(x(t),y(t))\neq (x(s),y(s)) \text{ for }s\neq t$.
 \item[(v)] $\sqrt{|\dot{x}(t)|^2+|\dot{y}(t)|^2}= 1$ almost everywhere.
\end{itemize}
Note that here condition (i) and (v) mean that we parameterize $\G$ by unit length; condition (ii) implies two endpoints of $\G$ belong to $x$-axis; condition (iv) tells that $\G$ does not touch itself and will always "point from left to right". Now we define $\Om_\G$ as the region enclosed by $\Gamma$ and $x$-axis. Note that so far $\Om_\G$ may not be a simply connected region since $\Gamma(t)$ may touch $x$-axis at some other point between two endpoints. However, we will show later in Lemma \ref{mini} that for a minimizer, $\Om_\G$ has to be simply connected.

The boundary condition for director $n$ on $\pa\Om_\G=\{(x,0): x\in[-a,a]\}\cap \Gamma$ is given by
\begin{align*}
& n(x,y)=(1,0)\text{ on } \{(x,0):x\in [-a,a]\},\\
& n(x(t),y(t))=(x'(t),y'(t)) \text{ on }(x(t),y(t))\in \Gamma.
\end{align*}
Note that in 2D, the unit vector can be determined by an angle function $\Theta$ according to $n_1=\cos{\Theta}, n_2=\sin{\Theta}$. We will work with this angle function $\Theta$. Then the corresponding boundary condition for $\Theta$ is
\begin{equation}
\label{bdycondition}
\begin{split}
&\Theta(x,y)=0\quad \text{ on } \{(x,0):x\in [-a,a]\}\\
&\Theta(x(t),y(t))=\arcsin{y'(t)} \quad  \text{ on }(x(t),y(t))\in \Gamma.
\end{split}
\end{equation}
Now we are ready to define the following admissible set for $\G$:
\begin{align*}
\mathcal{G}_v:=&\{\G \text{ satisfies condition (i--v)}, \text{ and }\Theta\big|_{\pa\Om_\G} \text{ has a harmonic extension }\Theta \text{ defined in }\overline{\Om_\G}\\
& \text{ such that }\int_{\Om_\G}|\na \T|^2\,dxdy<\infty \text{ and }|\Om_\G|=v \}
\end{align*}
Here $v$ is a positive constant representing the volume of $\Om_\G$. Figure \ref{figure1} shows our assumptions on $\G,\,\Om_\G$ and the tangential anchoring condition for $\T$.
\begin{figure}[h]
\centering
\begin{tikzpicture}[x=1.5cm,y=1.2cm]
  \clip (-5.7,-0.7) rectangle (5,4.0);
  \draw[black,thick] (-4.5,0) node[below]{$(-a,0)$}
  .. controls +(right:1.5cm) and +(left:2cm) .. node[above=1cm, right=0.2cm] {\Large$|\Omega_{\Gamma}|=v$} (-0.5,0)
    \foreach \p in {0,20,...,100} {
    node[sloped,inner sep=0.5cm,above,pos=\p*0.01,
    anchor=south west,
    minimum height=(10+\p)*0.03cm,minimum width=(10+\p)*0.03cm]
    (N0 \p){}
  }
  .. controls +(right:1.5cm) and +(left:2cm) .. (3.5,0)node[below]{$(a,0)$}
    \foreach \p in {0,20,...,100} {
    node[sloped,inner sep=0.5cm,above,pos=\p*0.01,
    anchor=south west,
    minimum height=(10+\p)*0.03cm,minimum width=(10+\p)*0.03cm]
    (N5 \p){}
    }
  ;
  \draw[black,thick] (-4.5,0)
  .. controls +(right:3cm) and +(left:2cm) .. (-2.2,2.8) node[above right] {$\Gamma=\{(x(t),y(t))\}$}
  \foreach \p in {0,20,...,100} {
    node[sloped,inner sep=0cm,above,pos=\p*0.01,
    anchor=south west,
    minimum height=(10+\p)*0.03cm,minimum width=(10+\p)*0.03cm]
    (N \p){}
  }
  .. controls +(right:2cm) and +(left:1.5cm) .. (-0.1,0.9)
  \foreach \p in {0,20,...,100} {
    node[sloped,inner sep=0cm,above,pos=\p*0.01,
    anchor=south west,
    minimum height=(110-\p)*0.03cm,minimum width=(110-\p)*0.03cm]
    (N2 \p){}
  }
    .. controls +(right:1.2cm) and +(left:1.2cm) .. (1.2,2.0) node[above=0.1cm] {$\Theta=\arctan{y'(t)}$}
  \foreach \p in {0,20,...,100} {
    node[sloped,inner sep=0cm,above,pos=\p*0.01,
    anchor=south west,
    minimum height=(110-\p)*0.03cm,minimum width=(110-\p)*0.03cm]
    (N3 \p){}
  }
      .. controls +(right:2cm) and +(left:2.6cm) .. (3.5,0)
  \foreach \p in {0,20,...,100} {
    node[sloped,inner sep=0cm,above,pos=\p*0.01,
    anchor=south west,
    minimum height=(110-\p)*0.03cm,minimum width=(110-\p)*0.03cm]
    (N4 \p){}
  }
  ;
    \foreach \p in {0,20,...,100} {
  \draw[-latex,color=black] (N0 \p.south west)
    -- ($(N0 \p.south west)!0.8cm!(N0 \p.south east)$);
  }
    \foreach \p in {0,20,...,100} {
  \draw[-latex,color=black] (N5 \p.south west)
    -- ($(N5 \p.south west)!0.8cm!(N5 \p.south east)$);
  }
  \foreach \p in {0,20,...,100} {
  \draw[-latex,color=black] (N \p.south west)
    -- ($(N \p.south west)!0.8cm!(N \p.south east)$);
  }
  \foreach \p in {0,20,...,100} {
  \draw[-latex,color=black] (N2 \p.south west)
    -- ($(N2 \p.south west)!0.8cm!(N2 \p.south east)$);
  }
    \foreach \p in {0,20,...,100} {
  \draw[-latex,color=black] (N3 \p.south west)
    -- ($(N3 \p.south west)!0.8cm!(N3 \p.south east)$);
  }
    \foreach \p in {0,20,...,100} {
  \draw[-latex,color=black] (N4 \p.south west)
    -- ($(N4 \p.south west)!0.8cm!(N4 \p.south east)$);
  }
\end{tikzpicture}
\caption{Curve $\G$, domain $\Om_\G$ and $\T$}
\label{figure1}
\end{figure}

To this end, we consider the following variational problem
\begin{problem}{P}
Find $\G\in \mathcal{G}_v$ that minimizes the following functional
\beq\label{energy}
E(\G)=\int_{\Om_\G}|\na\T|^2\,dxdy+l(\G),
\eeq
where $\Theta$ is determined by $\G$ in the following way
\begin{equation*}
\begin{cases}
\Delta\Theta=0,& \text{in }\Om_\G,\\
\Theta\big|_{\pa\Om_\G}& \text{ is defined as in \eqref{bdycondition}}.
\end{cases}
\end{equation*}
\end{problem}

We will study the existence and properties of global energy minimizers of Problem P in the rest of the article. In Section 2 we prove various geometric properties of $\G$ and $\Om_\G$. When the energy $E(\G)$ is finite (not necessarily a minimizer), we show that $\G$ is a vanishing chord-arc curve and $\nu\in \mathrm{VMO}$ on $\G$. Moreover, the arc-length parameterization $(x(t),y(t))$ belongs to $H^{3/2}(0,l)$. As a consequence, the function $\T$ defined on $\bar{\Om}_\G$ can be extended to a $H^1$ function on $\mathbb{R}^2$ according to the classical theory on the relationship of quasidisks and Sobolev extention domains. The existence of a global minimizer for Problem P is established in Section 3. The proof relies heavily on the properties proved in Section 2. We also show that $\G$ and $x$-axis will form two cusps near two intersection points. Under the assumption that $\G$ can be written as the graph of a $C^1$ function $f$, the Euler-Lagrange equation for $\G$ is also derived. Finally in Section 4 we study asymptotic profiles of $\G$ when the volume $v$ tends to be very large or small. We would like to point out that this article just represents an initial investigation of Problem P, and there are many open problems to be studied in the future.

\textbf{Acknowledgement}
The research of the authors are partially supported by an NSF grant DMS1955249.

\section{Geometric properties of $\G$ and $\Om_\G$}
\subsection{Sobolev extension domain}
We assume $v=1$ throughout this section. And if there exists a energy minimizer for Problem P, we denote it by $\G_m$. We further write the corresponding $\Om_{\G_m}$ and $\Theta$ function as $\Om_m$ and $\T_m$.  We start with the observation that

\vspace{2mm}
\textbf{Claim}: $\mathcal{G}_1$ is not empty. There is at least one smooth curve $\G\in\mathcal{G}_1$.
\vspace{2mm}

Actually we can find a smooth curve $\Gamma_0\in \mathcal{G}_1$ by directly constructing a curve $\G_0$. Let $\G_0$ be the graph of function $f_0(x)=\f{\cos{x}+1}{2\pi}, \,x\in[-\pi,\pi]$. By definition $\Omega_{\G_0}=\{(x,y): \,-\pi\leq x\leq \pi,\, 0\leq y\leq f_0(x)\}$ and we set $\Theta_0(x,y)=-\f{2\pi y}{\cos{x}+1}\arcsin{\f{\sin{x}}{2\pi}}$ for $(x,y)\in\Om_{\G_0}$. It is straightforward to check that $\Gamma_0$ satisfies the condition (i--v), $|\Om_{\G_0}|=1$, and $\Theta_0$ satisfies the boundary condition \eqref{bdycondition}. Then we compute the energy directly
\begin{align*}
E(\G_0)&=\int_{-\pi}^\pi \sqrt{1+\left(\f{d f_0}{dx}\right)^2}\,dx+\int_{-\pi}^\pi\int_0^{\f{\cos{x}+1}{2\pi}}|\pa_y\T_0|^2+|\pa_x \T_0|^2\,dydx\\
&=\int_{-\pi}^\pi\left\{ \sqrt{1+\f{\sin^2 x}{4\pi^2}}+\f{2\pi\left| \arcsin{(\f{\sin{x}}{2\pi})} \right|^2}{\cos{x}+1}+\f{\left|\f{\cos{x}(\cos{x}+1)}{\sqrt{4\pi^2-\sin^2{x}}}+\arcsin{(\f{\sin{x}}{2\pi})}\cdot \sin{x}\right|^2}{6\pi(\cos{x}+1)} \right\}\,dx\\
&\approx 12.65
\end{align*}
Therefore we have verified that $\G_0\in\mathcal{G}_1$. And if Problem P admits a global minimizer $\G_m$, then we get the following upper bound for the energy infimum:
\beqo
M:=E(\G_0)\geq E(\G_m)
\eeqo

The next lemma tells us that the minimizing curve $\G_m$, if exists, will not touch $x$-axis besides two endpoints, which implies $\Om_m$ is simply connected.

\begin{lemma}\label{mini}
If $\G_m$ is the global minimizer of $E(\G)$ among all $\G\in\mathcal{G}_1$ and it is parametrized by arc length as in condition (i--v), then for any $t\in(0,l(\G_m))$, we have $y(t)>0$.
\end{lemma}
\begin{proof}
We prove by contradiction. Assume $y(t_0)=0$ for some $t_0\in (0,l(G(m)))$, then the point $(x(t_0),y(t_0))$ cuts $\G_m$ into two parts, which are denoted by $\G_1$ and $\G_2$ respectively. We call the domain enclosed by $\G_i$ and $x-$axis as $\Om_i$ for $i=1,2$. Let $\al:=|\Om_1|$. We can further assume $\al\in (0,1)$ because if $\al=0$ or $1$, then one of $\G_i$ will coincide with $x-$axis which contradicts with the fact that $\G_m$ is a minimizer. Now we set
\begin{align*}
&\tilde{\G}_1=\f{1}{\sqrt{\al}}G_1,\quad \tilde{\Om}_1=\f{1}{\sqrt{\al}}\Om_1,\quad \tilde{\T}_1(\f{x}{\sqrt{\al}}, \f{y}{\sqrt{\al}})=\T_m(x,y)\text{ for }(x,y)\in\Om_1\\
&\tilde{\G}_2=\f{1}{\sqrt{1-\al}}G_2,\quad \tilde{\Om}_2=\f{1}{\sqrt{1-\al}}\Om_2,\quad \tilde{\T}_2(\f{x}{\sqrt{1-\al}}, \f{y}{\sqrt{1-\al}})=\T_m(x,y)\text{ for }(x,y)\in\Om_2
\end{align*}
Now we can easily check that for $i=1,2$,  $(\tilde{\G}_i,\tilde{\Om}_i,\tilde{\T}_i)$ are energy competitors (after some horizontal translations) for $(\G_m,\Om_m,\T_m)$. By basic scaling property we get
\begin{align*}
l(\tilde{\G}_1)&=\f{1}{\sqrt{\al}}l(\G_1), \quad l(\tilde{\G}_2)=\f{1}{\sqrt{1-\al}}l(\G_2),\\ \int_{\tilde{\Om}_i}&|\na\tilde{\T}_i|^2=\int_{\Om_i}|\na \T|^2 \text{ for }i=1,2
\end{align*}
The minimizing property yields
\begin{align*}
&\f{1}{\sqrt{\al}}l(\G_1)+\int_{\Om_1}|\na\T|^2\geq l(\G_1)+l(\G_2)+\int_{\Om_1}|\na\T|^2+\int_{\Om_2}|\na\T|^2,\\ &\f{1}{\sqrt{1-\al}}l(\G_2)+\int_{\Om_2}|\na\T|^2\geq l(\G_1)+l(\G_2)+\int_{\Om_1}|\na\T|^2+\int_{\Om_2}|\na\T|^2.
\end{align*}
Combining these two inequalities we arrive at
\begin{align*}
&l(\G_1)\geq \f{\sqrt{\al}}{1-\sqrt{\al}}l(\G_2)\geq \f{\sqrt{\al}}{1-\sqrt{\al}}\cdot\f{\sqrt{1-\al}}{1-\sqrt{1-\al}}l(\G_1)\\
\Rightarrow& \sqrt{\al(1-\al)}\leq (1-\sqrt{\al})(1-\sqrt{1-\al})\Rightarrow \al=0 \text{ or }1,
\end{align*}
which yields a contradiction.

\end{proof}

Now we want to prove some geometric properties of $\G\in
\mathcal{G}_1$ (not necessarily a minimizer). The next statement says that for any three points on $\Gamma$, they are supposed to satisfy a reversed triangle inequality, with a constant depending on $E(\G)$.
\begin{lemma}\label{reversedtriangleineq}
If $\G\in \mathcal{G}_1$ and $E(\G)\leq M$, then there exists a constant $C=C(M)$ such that for any three points $z_1=(x(t_1),y(t_1))$, $z_2=(x(t_2),y(t_2))$ and $z_3=(x(t_3),y(t_3))$ on $\G$ such that $t_1<t_2<t_3$, it holds that
\beq\label{reversetri}
\max{\{\dist(z_1,z_2),\dist(z_2,z_3)\}}\leq C \dist(z_1,z_3).
\eeq
\end{lemma}
\begin{proof}
Assume $C$ is a large enough number (say larger than 100) which will be determined later. We simply write $(x(t_i),y(t_i))$ as $(x_i,y_i)$ for $i=1,2,3$. Without loss of generality, we assume $y_1\geq y_3$. Then for the value of $y_2$, there are three cases:
\begin{enumerate}
  \item $y_1\geq y_2\geq y_3$,
  \item $y_2\geq y_1\geq y_3$,
  \item $y_1\geq y_3\geq y_2$.
\end{enumerate}

The inequality \eqref{reversetri} for the first case is trivial, because by simple geometry we can get
\beqo
\max{\{\dist(z_1,z_2),\dist(z_2,z_3)\}}\leq  \dist(z_1,z_3).
\eeqo
Now we study the second case, and assume \eqref{reversetri} is false. By triangle inequality, we have
\beqo
\min{\{\dist(z_1,z_2),\dist(z_2,z_3)\}}\geq (C-1) \dist(z_1,z_3).
\eeqo
Therefore it holds that
\beqo
\min\{y_2-y_1,y_2-y_3\}\geq (C-2)|x_1-x_3|.
\eeqo

For convenience we assume $y_2=\max_{t\in(t_1,t_3)}y(t)$. Otherwise we can take $z_2$ to be the point with the maximum value of $y$ on $\G$ between $z_1$ and $z_3$. Note that such choice will not violate any of the above estimates.

We set the curve $\G$ between $z_1,z_2$ and $z_2,z_3$ as $\G_1,\G_2$, written as $\G_1:=\G_{z_1z_2},\, \G_2:=\G_{z_2z_3}$.  Set $l_i:=l(\G_i)$ for $i=1,2$. Also we reparametrize $\G_1$ and $\G_2$ as following
\begin{align*}
\G_1:=\{&(x(s), y(s)): s\in [0,l_1],\,( x(0),y(0))=(x_2,y_2),\,(x(l_1),y(l_1))=(x_1,y_1),\\
& x'(s)\leq 0,\,|x'(s)|^2+|y'(s)|^2=1\;a.e.\},\\
\G_2:=\{&(x(s), y(s)): s\in [0,l_2],\,( x(0),y(0))=(x_2,y_2),\,(x(l_2),y(l_2))=(x_3,y_3),\\
& x'(s)\geq 0,\, |x'(s)|^2+|y'(s)|^2=1\;a.e.\}
\end{align*}
Note that for such reparametrization, $\G_1$ starts at $z_2$ and ends at $z_1$, while $\G_2$ starts at $z_2$ and ends at $z_3$. Also we have $\Theta(x(s),y(s))=-\arcsin{y'(s)}$ on $\G_1$ and $\Theta(x(s),y(s))=\arcsin{y'(s)}$ on $\G_2$.

We first look at $\G_1$. Set
\begin{align*}
r(s)&:=\sqrt{|x(s)-x_2|^2+|y(s)-y_2|^2},\; s\in[0,l_1],\\
I(r)&:=\{s\in [0,l_1]: r(s)=r\}
\end{align*}
For $0<r<|z_1-z_2|$, the circle $\{|z-z_2|=r\}$ will intersect with $\G_1$ and therefore $I_r$ is not empty. By definition we have
\begin{align}
\label{yint} &\int_0^{l_1} y'(s)\,ds=y_1-y_2,\\
\label{xint} &\int_0^{l_1} x'(s)\,ds=x_1-x_2.
\end{align}
For $r(s)$, we can estimate its derivative by
\beqo
|r'(s)|=\left|\f{(x(s)-x_2)\cdot x'(s)+(y(s)-y_2)\cdot y'(s)}{r(s)}\right|\leq \sqrt{|x'(s)|^2+|y'(s)|^2} =1.
\eeqo
Then by coarea formula, we have
\beqo
l_1\geq \int_0^{l_1}|r'(s)|\,ds=\int_0^{|z_1-z_2|} \mathcal{H}^0(I_r)\,dr
\eeqo
This tells us that for almost every $r\in [0,|z_1-z_2|]$, $\mathcal{H}^0(I_r)$ is finite. Note that $\mathcal{H}^0$ is just the counting measure, and we will simply write it as $|I_r|$. Denote by $A$ the subset of $[0,l_1]$ such that for any $s\in A$, $r'(s)=0$. Again co-area formula gives
\beqo
0=\int_A |r'(s)|\,ds=\int_0^{|z_1-z_2|} \mathcal{H}^0(A\cap I_r)\,dr
\eeqo
So $A\cap I_r=\varnothing$ for a.e. $r\in [0,|z_1-z_2|]$. We define
\beqo
R_0:=\{ r\in[0,|z_1-z_2|]:\, |I_r| \text{ is finite, and }|r'(s)|>0\text{ for any }s\in I_r \}.
\eeqo
We have $m([0,|z_1-z_2|]\backslash R_0)=0$. For any $r\in R_0$, we pick a representative from $I_r$ in the following way:
\beqo
s^r=\min\{s: s\in I_r\}.
\eeqo
We define the following two subsets:
\begin{align*}
R_1&:=\{r\in[2[z_1-z_3],|z_1-z_2|]\cap R_0: x'(s^r)\leq-1/2\},\\ R_2&:=\left(\left[2|z_1-z_3|,|z_1-z_2|\right]\cap R_0\right)\backslash R_1.
\end{align*}
Note that $R^1$ corresponds to the part of curve on $\G_1$ where is not ``too vertical". Using co-area formula again, we get
\begin{align*}
\f{1}{2} m(R_1) &\leq \left|\int_{R_1} \f{d x(s^r)}{ds}\cdot\left|\f{dr(s^r)}{ds}\right|^{-1}\,dr\right|\\
&\leq |\int_{r^{-1}(R_1)} x'(s)\,ds|\\
&\leq |x_1-x_2|.
\end{align*}
As a consequence, we get
\beq\label{measureR2}
m(R_2)\geq |z_1-z_2|-2|z_1-z_3|-2|x_1-x_2|\geq |z_1-z_2|-4|z_1-z_3|.
\eeq
Now we make the following observation:
\begin{center}
For any $r\in R_2$, $y'(s^r)<-\f{\sqrt{3}}{2}$.
\end{center}
This is a consequence of definition of $R_2$ and $s^r$. Since $r\in R_2$, we have
\begin{align*}
y'(s^r)>\f{\sqrt{3}}{2}&\text{ or  }y'(s^r)<-\f{\sqrt{3}}{2},\\
y(s^r)<y_2-\sqrt{3}|z_1-z_3|&, \quad x_2-|z_1-z_3| \leq x(s^r)<x_2,\\
|y(s^r)-y_2|&>\sqrt{3}|x(s^r)-x_2|.
\end{align*}
We also have that $\f{dr(s^r)}{ds}> 0$ because $(x(s^r),y(s^r))$ is the first point that $\G_1$ touches $\{|z-z_2|=r\}$. If $y'(s^r)>\f{\sqrt{3}}{2}$, then
\beqo
r'(s^r)=\f{x'(s^r)(x(s^r)-x_2)+y'(s^r)(y(s^r)-y_2)}{r}<0,
\eeqo
which yields a contradiction. Therefore we have verified the observation.

Now we deal with $\G_2$ in the same way with several minor modifications. We can show that there exists a $R_3$ such that \begin{align*}
R_3\subset &\left[2|z_1-z_3|, |z_2-z_3|\right], \quad m(R_3)\geq |z_2-z_3|-4|z_1-z_3|,\\
 &\text{ and }\;\forall r\in R_3, \;y'(s_r)<-\f{\sqrt{3}}{2}.
\end{align*}
Here $s_r$ is the point that $\G_2$ first touches $\{|z-z_2|=r\}$.

We are now ready to derive a contradiction. Denoting $R:=R_2\cap R_3$, then we have
\beqo
R\subset \left\{r: 2|z_1-z_3|\leq r\leq  \min\{|z_1-z_2|,|z_2-z_3|\}\right\},\quad m(R)\geq  \min\{|z_1-z_2|,|z_2-z_3|\}-8|z_1-z_3|.
\eeqo
For any $r\in R$, $\G_1$ first intersects $\{|z-z_2|=r\}$ at $z_1(r):=(x(s^r),y(s^r))$ and $\G_2$ first intersects with $\{|z-z_2|=r\}$ at $z_2(r):=(x(s_r),y(s_r))$. The arc $\arc{z_1(r)z_2(r)}$ is contained in $\Om_\G$ because of the definitions of $s^r,s_r$. Moreover, $\Theta=-\arcsin{y'(s^r)}>\f{\pi}{3}$ at $z_1(r)$ and  $\Theta=\arcsin{y'(s_r)}<-\f{\pi}{3}$ at $z_2(r)$. Then we are ready to estimates the Dirichlet energy of $\T$ in $\Om_\G$,
\begin{align}
\label{computeenergy}\int_{\Om_\G}|\na \T|^2\,dxdy\geq & \int_0^{C|z_1-z_3|} \,dr\int_{\{|z-z_2|=r\}\cap \Om_\G} |\na \Theta(z_1+re^{i\theta})|^2 \\
\nonumber\geq & \int_{r\in R} \f{|\T(z_1(r))-\T(z_2(r))|^2}{\pi r} \,dr\\
\nonumber\geq & \int_{8|z_1-z_3|}^{(C-1)|z_1-z_3|} \f{4\pi}{9}\f{1}{r}\,dr\\
\nonumber\geq & \log{\f{C-1}{8}}.
\end{align}
By choosing $C$ satisfying $\log{\f{C-1}{8}}\geq 2M$, we arrive at a contradiction with the energy bound. Thus we proved \eqref{reversetri} for case (2).

For case (3) when $y_1\geq y_3\geq y_2$, the proof follows similar arguments. Assume $y_2=\min_{t\in[t_1,t_3]} y(t)$. We still call the curve between $z_1,z_2$ and $z_2,z_3$ as $\G_1,\G_2$ and reparametrize them as before. And $r(s)$, $I(r)$, $R_0$, $s^r$, $s_r$, $z_1(r)$, $z_2(r)$ are all defined in the same way. Recall that $z_1(r):=(x(s^r),y(s^r))\in \G_1$ and $z_2(r):=(x(s_r),y(s_r))\in\G_2$.  Similarly, we can find $R_2\subset [2|z_1-z_3|, |z_1-z_2|]$ such that for any $r\in R_2$, $\T(z_1(r))<-\f{\pi}{3}$. Also there exists $R_3\subset [2|z_1-z_3|, |z_1-z_3|]$ such that for $r\in R_3$, $\T(z_2(r))>\f{\pi}{3}$.

Now we claim that for any $r\in R:=R_2\cap R_3$, it holds that
\beq\label{gradientoncircle}
\int_{\{|z-z_2|=r\}\cap \Om_f} |\na \Theta(z_1+re^{i\theta})|^2 \geq \f{C_1}{r}.
\eeq
Here $C_1$ is a constant that can be chosen as $\f{\pi}{18}$. This is the place where case (3) differs from case (2), because in case (2) the set $\{|z-z_2|=r\}\cap \Om_\G$ is just the arc $\arc{z_1(r)z_2(r)}$. However in case (3), $\{|z-z_2|=r\}\cap \Om_\G$ is more complicated.
We prove the claim by discussing following three situations (see Figure 2):
\begin{enumerate}
  \item $\{|z-z_2|=r\}$ only intersects with $\G$ at $z_1(r), z_2(r)$ and doesn't intersect with $x$-axis. Since $\T(z_1(r))<-\f{\pi}{3}$ and $\T(z_2(r))>\f{\pi}{3}$, we have
      \beqo
      \int_{\{|z-z_2|=r\}\cap \Om_\G} |\na \Theta(z_2+re^{i\theta})|^2\geq \f1r\int_{z_1+re^{i\theta}\in \Om_f} |\f{\pa\T}{\pa\theta}|^2d\theta\geq \f{C_1}{r}.
      \eeqo
  \item $\{|z-z_2|=r\}$ only intersects with $\G$ at $z_1(r), z_2(r)$ and also intersects with $x$-axis at $z_3(r)$. Without loss of generality we can assume the arc $\arc{z_1(r)z_3(r)}$ is contained in $\Om_f$. Then since $\T=0$ on $x$-axis, we have $|\T(z_1(r))-\T(z_3(r))|\geq \f{\pi}{3}$, then we can verify (\ref{gradientoncircle}) by the same calculation.
  \item $\{|z-z_2|=r\}$ intersects with $\g_f$ at more than two points. Let $z_3(r)=(x_3(r),y_3(r))$ be another point of intersection besides $z_1(r), z_2(r)$. Without loss of generality we assume $x_3(r)<x_1(r)$.  In our construction we make sure that $r'(s^r)>0$ at $z_1(r)$. And we can assume $\arc{z_1(r)z_3(r)}\subset \bar{\Om}_f$.  Therefore we have $\T(z_3(r))\geq 0$ because at $z_3(r)$, $\G(t)$ is "leaving" the disk $\{|z-z_2|\leq r\}$ as $t$ increases. This implies that $|\T(z_1(r))-\T(z_3(r))|\geq \f{\pi}{3}$ Then (\ref{gradientoncircle}) follows immediately in the same way.
\end{enumerate}

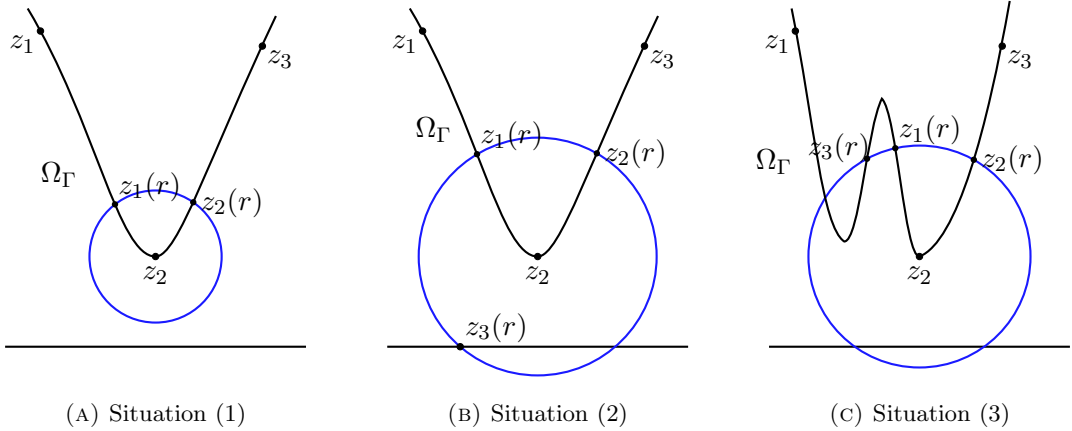
\begin{figure}[H]
\captionsetup[subfigure]{font=footnotesize}
\centering
\subcaptionbox{Situation (1)}[.3\textwidth]{
\begin{tikzpicture}
\clip(-2.2,-0.5) rectangle (2.2, 4.7);
\draw[thick] (-2,0)--(2,0);
\draw[thick, color=blue!90] (0,1.2)  circle [radius=25pt];
\path [name path=circ] (0,1.2)  circle [radius=25pt];
\path [name path=g1] (-1.7,4.5)
  .. controls  (-0.8,3) and (-0.5,1.2) ..  (0,1.2);
\path [name path=g2] (0,1.2)
  .. controls  (0.3,1.2) and (0.7,2.5) ..  (1.6,4.4);
\filldraw[black] (0,1.2) circle[radius=1.2pt] node[below] {$z_2$};
\filldraw[black] (-1.53,4.2) circle[radius=1.2pt] node[below left=-2pt] {$z_1$};
\filldraw[black] (1.42,4.0) circle[radius=1.2pt] node[below right=-2pt] {$z_3$};
\draw[black,thick] (-1.7,4.5)
  .. controls  (-0.8,3) and (-0.5,1.2) .. node[left=0.2cm] {$\Omega_\Gamma$} (0,1.2)
  .. controls  (0.3,1.2) and (0.7,2.5) ..  (1.6,4.4);
\fill[name intersections={of=g1 and circ, by=x1}] (x1) circle [radius=1.2pt] node[above right=-2.5pt ]{$z_1(r)$} ;
\fill[name intersections={of=g2 and circ, by=x2}] (x2) circle [radius=1.2pt] node[right=-1pt ]{$z_2(r)$} ;
\end{tikzpicture}
}
\subcaptionbox{Situation (2)}[.3\textwidth]{
\begin{tikzpicture}
\clip(-2.2,-0.5) rectangle (2.2, 4.7);

\draw[thick] (-2,0)--(2,0);
\draw[thick, color=blue!90] (0,1.2)  circle [radius=45pt];
\path [name path=circ] (0,1.2)  circle [radius=45pt];
\path [name path=g1] (-1.7,4.5)
  .. controls  (-0.8,3) and (-0.5,1.2) ..  (0,1.2);
\path [name path=g2] (0,1.2)
  .. controls  (0.3,1.2) and (0.7,2.5) ..  (1.6,4.4);
\filldraw[black] (0,1.2) circle[radius=1.2pt] node[below] {$z_2$};
\filldraw[black] (-1.03,0) circle[radius=1.2pt] node[above right=-2pt] {$z_3(r)$};
\filldraw[black] (-1.53,4.2) circle[radius=1.2pt] node[below left=-2pt] {$z_1$};
\filldraw[black] (1.42,4.0) circle[radius=1.2pt] node[below right=-2pt] {$z_3$};
\draw[black,thick] (-1.7,4.5)
  .. controls  (-0.8,3) and (-0.5,1.2) .. node[above left=0.3cm] {$\Omega_\Gamma$} (0,1.2)
  .. controls  (0.3,1.2) and (0.7,2.5) ..  (1.6,4.4);
\fill[name intersections={of=g1 and circ, by=x1}] (x1) circle [radius=1.2pt] node[above right=-2.5pt ]{$z_1(r)$} ;
\fill[name intersections={of=g2 and circ, by=x2}] (x2) circle [radius=1.2pt] node[right=-1pt ]{$z_2(r)$} ;
\end{tikzpicture}

}
\subcaptionbox{Situation (3)}[.3\textwidth]{
\begin{tikzpicture}
\clip(-2.2,-0.5) rectangle (2.2, 4.7);
\draw[thick] (-2,0)--(2,0);
\draw[thick, color=blue!90] (0,1.2)  circle [radius=42pt];
\path [name path=circ] (0,1.2)  circle [radius=42pt];
\path [name path=g1] (-0.5,3.3)
  .. controls  (-0.3,3.1) and (-0.2,1.2) ..  (0,1.2);
\path [name path=g2] (0,1.2)
  .. controls  (0.5,1.4) and (0.9,2.9) ..  (1.2,4.6);
\path [name path=g3] (-1,1.4)
  .. controls  (-0.8,1.4) and (-0.7,2.9) ..  (-0.5,3.3);
\filldraw[black] (0,1.2) circle[radius=1.2pt] node[below] {$z_2$};
\filldraw[black] (-1.65,4.2) circle[radius=1.2pt] node[below left=-2pt] {$z_1$};
\filldraw[black] (1.1,4.0) circle[radius=1.2pt] node[below right=-2pt] {$z_3$};
\draw[black,thick] (-1.7,4.5)
  .. controls  (-1.4,3) and (-1.3,1.6) .. node[left=0.2cm] {$\Omega_\Gamma$} (-1,1.4)
  .. controls  (-0.8,1.4) and (-0.7,2.9) ..  (-0.5,3.3)
  .. controls  (-0.3,3.1) and (-0.2,1.2) ..  (0,1.2)
  .. controls  (0.5,1.4) and (0.9,2.9) ..  (1.2,4.6) ;
\fill[name intersections={of=g1 and circ, by=x1}] (x1) circle [radius=1.2pt] node[above right=-2.5pt ]{$z_1(r)$} ;
\fill[name intersections={of=g2 and circ, by=x2}] (x2) circle [radius=1.2pt] node[right=-1pt ]{$z_2(r)$} ;
\fill[name intersections={of=g3 and circ, by=x3}] (x3) circle [radius=1.2pt] node[above left=-5pt ]{$z_3(r)$} ;
\end{tikzpicture}

}
\centering
\caption{Three different situations of $I_r$ and $\{|z-z_2|=r\}\cap \Om_\G$,  when $y_1\geq y_3>y_2$}
\label{fig2}
\end{figure}

With (\ref{gradientoncircle}), we can repeat the computation in \eqref{computeenergy} and finally verifies \eqref{reversetri} for case (3). This completes our proof of Lemma\ref{reversedtriangleineq}.

\end{proof}

A direct consequence of Lemma \ref{reversedtriangleineq} is that $\G$ satisfy the following ``two point condition":
\beqo
\text{For any }z_1,z_2\in\G\text{, set }\g\text{ be the arc of }\G\text{ between }z_1,z_2\text{, then }\mathrm{diam} \,\g\leq C |z_1-z_2|.
\eeqo

It is proved by Ahlfors in \cite{ahlfors} that a Jordan curve is a quasicircle if and only if it satisfies the ``two point condition". A quasicircle is the image of the unit circle $\mathbb{T}$ under a quasiconformal mapping of the complex plane onto itself. And in 2D a quasidisk (domain enclosed by a quasicircle) is equivalent to a Sobolev extension domain, see \cite{Jones}. However in our problem, $\pa\Om_\G=\G\cap \{(x,0):x\in[-a,a]\}$ is not a quasicircle because near two endpoints $(-a,0)$ and $(a,0)$ the ``two point condition" will be violated.
The next lemma says that even though we cannot directly use the property of a quasidisk, we can still extend $\T$ to the whole plane with a uniform control on its $H^1$-norm.

\begin{lemma}\label{extension}
(Extension domain) Assume $\G\in\mathcal{G}_1$ and $E(\G)\leq M$, then there exists a constant $C_2$ that depends on $M$ such that $\T|_{\overline{\Om}_\G}$ can be extended to the whole plane with a norm control
\beqo
\|\T\|_{H^1(\mathbb{R}^2)}\leq C_2 \|\T\|_{H^1(\Om_\G)}.
\eeqo
\end{lemma}
\begin{proof}
The idea of the proof is to add a rectangle to $\Omega_\G$ to make the combined domain a quasidisk. Assume the intersection points of $\G$ and $x$-axis are $(-a,0)$ and $(a,0)$, we set
\beqo
D_1:=\{(x,y): -a< x< a,\, -a <y\leq 0\},\quad D_\G=\Omega_\G\cup D_1.
\eeqo

We claim that $D_\G$ is a quasidisk. For any two points $z_1,z_2\in\pa D_\G$, if $z_1,z_2\in\G$, then Lemma \ref{reversedtriangleineq} says they satisfy the ``two point condition". If $z_1,z_2\in\pa D_1\cap \pa D_\G$, they automatically satisfy the ``two point condition" because a rectangle is a quasidisk. So we are left with the case $z_1\in \G, \, z_2\in \pa D_1\cap \pa D_\G$. In such case, there are two situations:
\begin{enumerate}
\item $z_2\in \{(x,-a): -a\leq x\leq a\}$. Then $|z_1-z_2|\geq a$ and Lemma \ref{reversedtriangleineq} implies that $\mathrm{dia}(D_\G)\leq (C+2)a$ where $C$ is the constant in \eqref{reversetri}. So the ``two point condition" holds for this situation.
\item $z_2\in \{(-a,y): -a\leq y\leq 0\}\cup \{(a,y):-a\leq y\leq 0\}$. Without loss of generality we assume $z_2=(-a, y)$ and set $z_0:=(-a,0)$. Let $\g$ be the arc of $\G$ between $z_0$ and $z_1$, then we get
    \beqo
    |z_1-z_2|\geq \max\{|y|, |z_1-z_0|\},\quad \mathrm{diam}(\overline{z_2z_0}\cup \g)\leq |y|+\mathrm{diam}(\g),
    \eeqo
    where $\overline{z_2z_0}$ is the line segment between $z_0$ and $z_2$. One can easily shows that $\mathrm{diam}(\overline{z_2z_0}\cup\g)\leq (C+1)|z_1-z_2|$, since $\mathrm{diam}\,{\g}\leq C|z_1-z_0|$ by \eqref{reversetri}. The ``two point condition" is verified.
\end{enumerate}

Therefore by Ahlfors' result, we prove the claim. Next we can trivially extend $\T$ to $D_\G$ by let $\T(x,y)\equiv 0$ for $(x,y)\in D_1$, because $\T$ vanishes on $\{(x,0): -a\leq x\leq a\}$. Obviously, $\|\T\|_{H^1(D_\G)}=\|\T\|_{H^1(\Om_\G)}$. Moreover, we can further extend $\T$ to the whole plane $\mathbb{R}^2$, since a 2D domain is a quasidisk if and only if it is a Sobolev extension domain. And there exists $C_2:=C_2(M)$ such that
\beqo
\|\T\|_{H^1(\mathbb{R}^2)}\leq C_2\|\T\|_{H^1(\Om_\G)}.
\eeqo

This completes our proof of Lemma \ref{extension}.

\end{proof}

\subsection{$\G$ is a chord-arc curve}
We will give more geometric properties of $\G$ by showing that it is a chord-arc curve,  which means the length of the chord is comparable with the length of the arc (see \cite{JK} for a detailed discussion on chord-arc curves).

\begin{proposition}\label{chordarc}
Let $\G\in \mathcal{G}_1$ and $E(\G)\leq M$. There exists a constant $C_3(M)$ such that for any two points $z_1,z_2\in\G$ and the arc $\g:=\G_{z_1z_2}$, we have
\beqo
l(\g)< C_3|z_1-z_2|.
\eeqo
In other words, $\G$ is a chord-arc curve.
\end{proposition}

\begin{proof}
First assume $C_3$ is a very large number that will be determined later. We prove by contradiction. Suppose $\mathcal{H}_1(\g)=c|z_1-z_2|$ for some $c>C_3$, our goal is to show that the Dirichelt energy ``generated" by this part of boundary will be very large, which contradicts to the uniform bounds of Dirichlet energy ($E(\G)\leq M$). The basic idea can be roughly stated as following: if the length of curve is way too long compared with the chord length, then there will be lots of fluctuations of the curve, which will lead to large energy. The co-area formula will be used repeatedly in the proof.

Since 2D Dirichlet energy is scaling invariant, we simply let $|z_1-z_2|=1$, and reparametrize the curve $\g$ in the following way:
\begin{align*}
\g=\{&\left(x(t),y(t)\right):\, x(t),y(t)\in \mathrm{AC}[0,c];\, x'(t)\geq 0;\, z_1=(x(0),y(0))=(0,0),\\
& z_2=(x(c),y(c))=(a,\pm\sqrt{1-a^2})\,\text{ for some }a\in(0,1]; |x'(t)|^2+|y'(t)|^2=1\,\mathrm{a.e.} \}.
\end{align*}
According to Lemma \ref{reversedtriangleineq}, we have $|y(t)|\leq C$ for $t\in[0,c]$, where $C$ is the constant in \eqref{reversetri}. So there exists $Y_1\leq 0,Y_2\geq 0$ such that $|Y_i|\leq C$ for $i=1,2$ and $\min y(t)=Y_1, \, \max y(t)=Y_2$. Furthermore we have
\beq\label{rectangle}
\Gamma\subset \{0\leq x\leq a, Y_1\leq y\leq Y_2\}=:Q.
\eeq
We have the following upper bound for the energy
\beqo
\int_{\Omega_\g}|\na \T|^2\,dxdy\leq M, \text{ where }\Omega_\g:=Q\cap\{(x,y)\text{ below }\g\}.
\eeqo
Note that by Lemma \ref{extension},  we can extend the domain of $\Theta$ to all of
$Q$ such that
\beqo
\int_{Q}|\na \T|^2\,dxdy\leq C_2M=:C_4.
\eeqo
Here this constant $C_4$ only depends on $M$. Also we make the following definitions:
\begin{align*}
&T_s:=\{t\in[0,c]: y(t)=s\},\qquad \forall s\in[Y_1,Y_2]\\
&U:=\{s\in[Y_1,Y_2]: |T_s| \text{ is infinite} \},\quad W:=\{t\in[0,c]: y(t)\in U\}.\\
&S:=\{s\in[Y_1,Y_2]: |T_s|=1\},\qquad A:=\{t\in[0,c]: y(t)\in S\}.\\
&V:=[Y_1,Y_2]\backslash (S\cup U),\qquad B:=[0,c]\backslash (A\cup W).
\end{align*}
Here $|\cdot|$ denotes the cardinality of a set.

\par
Now  we have set up all the assumptions and are ready to derive a contradiction. First we point out several elementary observations:
\begin{itemize}
\item[(i)] The following estimate holds:
\beq\label{deriofy}
\int_0^c |y'(t)|^2\,dt=c-\int_0^c|x'(t)|^2\,dt\geq c-\int_0^c |x'(t)|dt\geq c-1.
\eeq
\item[(ii)] If $Y_2>1$, then for any $s\in[1,Y_2)$, we have $|T_s|\geq 2$ by mean value theorem for continuous function. Similarly, if $Y_1< -1$, for any $s\in(Y_1,-1]$, it holds that $|T_s|\geq 2$. In other words, we have
    \beqo
    S\subset [-1,1]\cap \{Y_1,Y_2\}.
    \eeqo
\item[(iii)] We can estimate the measure of $A$ by
\begin{align}
\label{measureofA}m(A)&=\int_A \sqrt{|x'(t)|^2+|y'(t)|^2}\,dt\\
\nonumber &\leq \int_A |x'(t)|\,dt+\int_A |y'(t)|\,dt\\
\nonumber &\leq 1+\int_S \mathcal{H}^0(T_s)\,ds \text{ (by coarea formula)}\\
\nonumber &=1+|S|\leq 3.
\end{align}
\item[(iv)] By co-area formula one can easily check that
\beqo
m(U)=0,\quad \int_{W} |y'(t)|\,dt=0
\eeqo
\end{itemize}
\par
For any $s\in V$, by definition we have $2\leq T_s< \infty$, we want to derive a lower bound for the following quantity:
\beqo
E(s):=\int_0^a \left|d_x\Theta(x,s)\right|^2\,dx.
\eeqo
Assume $T_s=\{t_1,\dots,t_n\}$ for some $n\geq 2$, and by definition we have
\beqo
\sin(\Theta(x(t_i),s))=y'(t_i), \text{ for } i=1,\dots,n.
\eeqo
An easy observation is that for each two adjacent points, say $t_i$ and $t_{i+1}$,
\beqo
y'(t_i)\cdot y'(t_{i+1})\leq 0.
\eeqo
We deduce that
\beqo
|\Theta(x(t_i),s)-\Theta(x(t_{i+1}),s)|\geq |\sin(\T(x(t_i),s))-\sin(\T(x(t_{i+1}),s)|=|y'(t_{i})|+|y'(t_{i+1})|.
\eeqo
Then we estimate $E(s)$ as following
\begin{align}\label{energyline}
E(s)&=\int_0^1 |d_x(\T(x,s))|^2\,dx \\
\nonumber &\geq \sum\limits_{i=1}^{n-1} \int_{x(t_{i})}^{x(t_{i+1})} |d_x(\T(x,s))|^2\,dx\\
\nonumber &\geq \sum\limits_{i=1}^{n-1} \int_{x(t_i)}^{x(t_{i+1})} \left|\f{\Theta(x(t_{i+1}),s)-\T(x(t_i),s)}{x(t_{i+1})-x(t_{i})}\right|^2\,dx\\
\nonumber &\geq \sum\limits_{i=1}^{n-1} \f{(|y'(t_i)|+|y'(t_{i+1})|)^2}{x(t_{i+1})-x(t_i)}\\
\nonumber \text{(Cauchy-Schwarz) }&\geq \f{(\sum\limits_{i=1}^n |y'(t_i)|)^2}{x(t_n)-x(t_1)}\\
\nonumber &\geq \sum\limits_{t\in T_s} |y'(t)|^2.
\end{align}

Using coarea formula, we get
\beq\label{coarea}
\int_{B}|y'(t)|^3\,dt=\int_{U}\left(\sum\limits_{t\in T_s} |y'(t)|^2\right)\,ds
\eeq
Then by estimating the Dirichlet energy inside $Q$ using \eqref{energyline} and \eqref{coarea}, we have that
\begin{align}
\label{energyest}C_4&\geq \int_Q |\na \T(x,y)|^2\,dxdy\\
\nonumber &\geq \int_V E(s)\,ds\\
\nonumber &\geq \int_{V}\left(\sum\limits_{t\in T_s} |y'(t)|^2\right)\,ds\\
\nonumber &=\int_B |y'(t)|^3\,dt
\end{align}
H\"{o}lder inequality further implies that
\beq\label{holder}
\left(\int_B |y'(t)|^3\,dt\right)\geq \left( \int_B|y'(t)|^2\,dt \right)^{3/2}\cdot m(B)^{-1/2}.
\eeq
By \eqref{deriofy} and \eqref{measureofA}, we have
\beq\label{deriyonB}
\int_B |y'(t)|^2\,dt\geq c-1-\int_A |y'(t)|^2\,dt-\int_W|y'(t)|^2\,dt\geq c-4.
\eeq
As a result, combining \eqref{energyest}, \eqref{holder} and \eqref{deriyonB} leads to
\beqo
C_4\geq \f{(c-4)^{3/2}}{c^{1/2}},
\eeqo
which yields a contradiction if we choose the constant $C_3$ to be large enough at first (recall that $c$ is a real number larger than $C_3$). Now that since $C_4$ only depends on $M$, $C_3$ also only depends on $M$. This completes our proof of Proposition \ref{chordarc}.
\end{proof}

Actually, we can examine the chord-arc property of $\G$ more closely and prove that it is indeed a vanishing chord-arc (also called ``approximately smooth'') curve, which is the following lemma.

\begin{proposition}\label{vanishingchordarc}
Let $\G\in \mathcal{G}_1$ and $E(\G)\leq M$. For any $\e>0$, there exists a $r=r(\e,\G)$ such that for any two points $z_1,z_2\in\G$ that satisfy $|z_1-z_2|\leq r$, we have
\beqo
l(\g)\leq (1+\e)|z_1-z_2|,
\eeqo
where $\g=\G_{z_1z_2}$. That is to say, $\G$ is a vanishing chord-arc (approximately smooth) curve.
\end{proposition}

\begin{proof}
The technique here will be very similar to the proof of Proposition \ref{chordarc}. We will only present our main ingredients and omit some computational details. Take any $z_1,z_2\in \G$ such that $|z_1-z_2|=r$. By Lemma \ref{reversedtriangleineq}, $\g=\G_{z_1z_2}$ must be contained in a rectangle domain $Q_{z_1z_2}$ with width $r$ and length $2Cr$ (see \eqref{rectangle} for the existence of such rectangle). Let $\T$ be the angle function determined by $\G$, we again extend $\T$ to $\mathbb{R}^2$ such that $\int_{\mathbb{R}^2}|\na\T|^2\,dx<C_2M$. Then for any $\d>0$, there exists a $\sigma>0$ such that
\beqo
\int_{E} |\na\T|^2\,dx\leq \d, \quad \text{whenever }|E|<\sigma.
\eeqo
Therefore as $r\ri 0$, the Dirichlet energy of $\T$ inside $Q_{z_1z_2}$ will go to zero. The convergence rate doesn't depend on the choice of $z_1,z_2$ but only depends on their distance $r$. As a consequence, in order to prove the lemma, we only need to prove the following statement: for any $\e>0$, there exists a constant $C(\e)>0$ such that
\beq\label{vanish}
\int_{Q_{z_1z_2}}|\na\T|^2\,dx\geq C(\e),\, \text{ whenever }\,l(\g)\geq (1+\e)|z_1-z_2|
\eeq
Now we fix $\e>0$. By scaling invariant property, we assume without loss of generality that $|z_1-z_2|=1$ and $l(\g)= 1+\e$, $z_1=(0,0)$, $z_2=(\cos\al,\sin\al)$ for some $\al\in(0,\f{\pi}{2})$. Note that here $|\al|\neq \f{\pi}{2}$, otherwise $\g$ would be line segment orthogonal to $x-$axis and $l(\g)=1$. We parameterize $\g$ as
\beqo
\g:=\{(x(t),y(t)): (x(0),y(0))=z_1,\, (x(1+\e),y(1+\e))=z_2,\, |x'(t)|^2+|y'(t)|^2=1, \text{ a.e.}\}
\eeqo
Set
\beqo
h(t):=\cos\al \cdot y(t)-\sin\al \cdot x(t),\quad g(t):=\sin\al\cdot y(t)+\cos\al\cdot x(t).
\eeqo
We have
\beq\label{xyh}
\int_0^{1+\e} x'(t)\,dt=\cos\al,\quad \int_0^{1+\e}y'(t)\,dt=\sin\al,\quad h(0)=h(1+\e)=0.
\eeq
Set $h_{max}=\max\limits_{0\leq t\leq 1+\e}h(t)$ and $h_{min}=\min\limits_{0\leq t\leq 1+\e}h(t)$, and define
\begin{align*}
T_s&:=\{t\in [0,1+\e]: h(t)=s\} \text{ for }s\in [h_{min}, h_{max}],\\
B&:=\{t\in [0,1+\e] : 2\leq T_{h(t)}<\infty \}.
\end{align*}
Obviously for any $s\in (h_{min}, h_{max})$ we have $|T_s|\geq 2$. Also we should deduct the subset of $[h_{min}, h_{max}]$ such that $|T_s|$ is infinite (see the definition of set $U,W$ in the proof of Proposition \ref{chordarc}). But from the argument in the proof of Proposition \ref{chordarc} we know it is a measure zero set and won't affect our computation, so we may simply assume $2\leq|T_{s}|<\infty$ for any $s\in (h_{min}, h_{max})$.

We discuss in two cases.
\begin{case}
If $g'(t)=\sin\al\cdot y'(t)+\cos\al \cdot x'(t)\geq 0$ for a.e. $t\in [0,1+\e]$. We calculate in the same way as \eqref{energyline}, \eqref{coarea} and \eqref{energyest} (the only difference is we replace $y(t)$ with $h(t)$) and obtain
\beq\label{energyestcase1}
\int_{Q_{z_1z_2}}|\na \T|^2\,dxdy\geq \int_{Q_{z_1z_2}}|\pa_g \T|^2\,dxdy\geq  C\int_B | h'(t)|^3\,dt
\eeq
Here $C$ is a positive constant only depend on $M$. Also, from \eqref{xyh} and the assumption $g'(t)\geq 0$ we have
\begin{align}\label{gtcase1}
&\int_0^{1+\e} |\sin\al\cdot y'(t)+\cos\al\cdot x'(t)|^2\,dt\\
\nonumber =&\int_0^{1+\e} |\sin\al\cdot y'(t)+\cos\al\cdot x'(t)|\cdot (\sin\al\cdot y'(t)+\cos\al\cdot x'(t)) \,dt\\
\nonumber \leq &\int_0^{1+\e} (\sin\al\cdot y'(t)+\cos\al\cdot x'(t))\,dt=1
\end{align}
and
\beqo
\int_0^{1+\e} |\sin\al\cdot y'(t)+\cos\al\cdot x'(t)|^2+|\cos\al \cdot y'(t)-\sin\al \cdot x'(t)|^2 \,dt=1+\e
\eeqo
The above two inequalities imply
\beqo
\int_0^{1+\e} |h'(t)|^2\,dt\geq \e.
\eeqo
By co-area formula we know that the set where $h(t)= h_{min}$ or $h_{max}$ contributes nothing in  the above integral, so we have
\beq\label{htesticase1}
\int_B |h'(t)|^2\,dt\geq \e.
\eeq

Then we combine \eqref{energyestcase1}, \eqref{htesticase1} and H\"{o}lder inequality to conclude that
\beq\label{hintegral}
\int_{Q_{z_1z_2}}|\na \T|^2\,dxdy\geq C\f{\e^{3/2}}{(1+\e)^{1/2}}.
\eeq
\end{case}

\begin{case}
Assume $g'(t)\geq 0 $ doesn't hold almost everywhere, then we may lose the estimate \eqref{gtcase1}. If we still have $\int_0^{1+\e}|g'(t)|^2\leq 1$, then all the estimates in Case 1 still hold and there is nothing to prove. So we assume
\beq\label{gtcase2}
\int_0^{1+\e} |g'(t)|^2\,dt =1+\d\quad \text{for some } 0<\d\leq \e.
\eeq
Then
\beq\label{htcase2}
\int_B |h'(t)|^2\,dt=\e-\d.
\eeq
Then the same computation leads to
\beq\label{hintegral2}
\int_{Q_{z_1z_2}}\left| \pa_g Q\right|^2 \,dxdy\geq C\f{(\e-\d)^{3/2}}{(1+\e)^{1/2}}.
\eeq
Now we set
\begin{align*}
g_{max}&:=\max\limits_{0\leq t\leq 1+\e}g(t),\quad g_{min}:=\min\limits_{0\leq t\leq 1+\e}g(t),\\
T^2_s&:=\{t\in [0,1+\e]: g(t)=s\} \text{ for }s\in [g_{min}, g_{max}],\\
B_2&:=\{t\in [0,1+\e] : 2\leq |T^2_{g(t)}|<\infty \},\\
A_2&:=\{t\in [0,1+\e]: g'(t)<0\}.
\end{align*}
Again we can ignore the set where $|T^2_{g(t)}|=\infty$ since it may lead to more complicated notations but won't affect any of our estimates. So we assume for every $t\in [0,1+\e]$, we have $|T^2_{g(t)}|<\infty$. Then simple geometry tells us that $A_2\subset B_2$. Also, since $\int_0^{1+\e} g'(t)\,dt=1$ and $\int_0^{1+\e} |g'(t)|\geq 1+\d$, it holds that
\beq\label{gtestcase2}
\int_{B_2}|g'(t)|\,dt\geq \int_{A_2}|g'(t)|\,dt\geq  \f{\d}{2}.
\eeq
\end{case}
Similar techniques in \eqref{energyline}, \eqref{coarea} and \eqref{energyest} imply that
\beq\label{gintegral2}
\int_{Q_{z_1z_2}}\left| \pa_h Q\right|^2 \,dxdy\geq C \int_{B_2} |g'(t)|^3\,dt \geq C \f{(\d/2)^3}{(1+\e)^{2}}. \text{ (H\"{o}lder inequality)}
\eeq
where $C$ is a constant that only depends on $M$.  We can combine this with \eqref{hintegral2} to get
\begin{align*}
&\int_{Q_{z_1z_2}} |\na Q|^2\,dxdy\\
\geq & \int_{Q_{z_1z_2}} |\pa_g Q|^2\,dxdy+\int_{Q_{z_1z_2}} |\pa_h Q|^2\,dxdy\\
\geq & C\left( \f{(\e-\d)^{3/2}}{(1+\e)^{1/2}}+ \f{(\d/2)^3}{(1+\e)^{2}}\right)\\
\geq & C(\e,M)
\end{align*}
This implies \eqref{vanish} and completes our proof of Proposition \ref{vanishingchordarc}.

\end{proof}

\begin{corol}
The normal vector $\nu$ along the curve $\Gamma$ belongs to $\mathrm{VMO}$ (vanishing mean oscillation space), i.e.
\beqo
\lim\limits_{r\ri 0} \left(\f{1}{l(B(x,r)\cap \G)}\int_{B(x,r)\cap \G}|\nu-\nu_{B(x,r)}| \,dl \right)=0 \;\text{ uniformly for }x\in\G,
\eeqo
where
\beqo
 \nu_{B(x,r)}=\f{1}{l(B(x,r)\cap \G)}\int_{B(x,r)\cap \G}\nu \,dl
\eeqo
\end{corol}
\begin{proof}
We already have that $\G$ is a vanishing chord-arc curve, also we can easily check that $\G$ is vanishing Reifenberg flat. Therefore we can direct apply results of Kenig $\&$ Toro \cite{kt} to conclude that $\nu\in\mathrm{VMO}$.
\end{proof}


\subsection{Weil-Petersson curve, $H^{3/2}$ characterization, and $\beta$-number}\label{weil-p}
Recall that a quasicircle is the image of the unit circle $\mathbb{T}$ under a quasiconformal map $f$ of $\BR^2$, e.g. a homeomorphism of the plane that is conformal outside the unit disk $\mathbb{D}$, whose dilatation $\mu$ satisfies $\|\mu\|_{L^\infty(\mathbb{D})}<1$. The collection of planar quasicircles is called the universal Teichm\"{u}ller space $T(1)$ and the metric is defined in terms of $\|\mu\|_\infty$. Takhtajan and Teo \cite{tt} defined a Weil-Petersson metric on universal Teichm\"{u}ler space $T(1)$ that makes it into a Hilbert manifold. A Weil-Petersson curve is the image of $\mathbb{T}$ under a quasiconformal map $f$ on the plane, and satisfies $|\mu|\in L^2(d\,A_\rho)$, where $A_\rho$ is hyperbolic area on $\mathbb{D}$. Another characterization for the Weil-Petersson curve is in terms of conformal mapping $f: \mathbb{D}\ri \Om$, where $\Om$ is the domain bounded by $\G$. $\G$ is a Weil-Petersson curve if and only if $(\log{f'})'\in L^2(\mathbb{D})$. In our problem, the curves in $\mathcal{G}_v$ resemble a lot to the Weil-Petersson curves in the following way. For a Weil-Petersson curve $\G$, let $f:\mathbb{D}\ri \Om$ be a conformal mapping. We focus on the boundary map $f:\mathbb{T}\ri \G$. Since $\log{f'}$ is in the Dirichlet space, we have that $\mathrm{arg}f'(z)$, as a function on $\mathbb{T}$, has a finite Dirichlet energy extension inside $\mathbb{D}$. One can check that for any $a\in\mathbb{T}$, it holds $\mathrm{arg}\,f'(a)=\mathrm{arg}\, \nu_\G(f(a))-2\pi a$, where $\nu_\G(f(a))$ is the outer normal vector of $\G$ at the point $f(a)$. Thus $\mathrm{arg}\,\nu_\G(b)-2\pi f^{-1}(b)$, as a function of $b\in \G$, has a finite Dirichlet energy inside $\Om$. Note that in our definition of $\mathcal{G}_v$, we require the $\theta=\mathrm{arg}\,\nu_\G-\f{\pi}{2}$ on the curve $\G$, and it has a finite energy extension inside $\Om_\G$. Such characterization is very similar to the Weil-Petersson curve. The difference is that in our case $\G$ is not a closed curve and the domain $\Om_\G$ is not a quasidisk.

In a recent work \cite{bishop}, Christopher Bishop gives 26 equivalent characterizations of the Weil-Petersson class. In particular, he shows that a curve $\G$ is Weil-Petersson if and only if it has arclength parameterization in $H^{3/2}(\mathbb{T})$, has finite M\"{o}bius energy or can be well approximated by polygons in some precise sense. Another equivalent characterization is that Weil-Peterssson curve has local curvature that is square integrable over all locations and scales, where local curvatures are measured using various quantities such as Peter Jone's $\beta$-numbers, conformal welding and Menger curvature. We will show that some of these function theoretic and geometric characterizations can be generalized to our curve $\G\in\mathcal{G}_v$, which greatly deepen our understanding of the class $\mathcal{G}_v$. The proof will follow Bishop's paper \cite{bishop} closely, with some necessary modifications.

Given $\G\in\mathcal{G}_v$, we denote the length of $\G$ by $l$. Let $z(t)=(x(t),y(t)):[0,l]\ri \G$ be the arc-length parameterization of $\G$, i.e. $\sqrt{x'(t)^2+y'(t)^2}=1$. Then $\G$ has the following properties.

\begin{proposition}\label{weil-peterson}
\begin{enumerate}
  \item The arc-length parameterization $z(t):[0,l]\ri \G$ is in the Sobolev space $H^{3/2}([0,l])$.
  \item Let $\nu$ be the normal vector, it holds that
  \beqo
  \int_\G\int_\G \left(\f{|\nu(z)-\nu(w)|}{|z-w|} \right)^2|dz|\,|dw|<\infty.
  \eeqo
\end{enumerate}
\end{proposition}
\begin{proof}
(1) As in the proof of Lemma \ref{extension}, we construct a quasidisk $D_\G$ which is the combination of $\Om_\G$ and a rectangle with length $2a$ and width $a$
\beqo
D_\G=\Om_\G\cup \{(x,y): -a<x<a,\;-a<y\leq 0\}.
\eeqo
The length of $\pa D_\G$ is $l+4a$. We denote the arc-length between $z\in\G$ and $(-a,0)$ as $l(z)$. Define the function $\phi:\mathbb{T}\ri\mathbb{T}$ as
\beq
\phi(\theta)=\begin{cases}
\mathrm{arg}\;\nu(z((l+4a)\theta)), & 0\leq \theta\leq \f{l}{l+4a},\\
\f{\pi}{2},& \f{l}{l+4a}<\theta<1.
\end{cases}
\eeq
One can easily check that in order to show $z(t)\in H^{3/2}([0,l])$, it suffices to show $\phi\in H^{1/2}(\mathbb{T})$. We also define an orientation preserving arclength parameterization $w:\mathbb{T}\ri \pa D_\G$, such that $|w'|=l+4a$, $w(0)=(-a,0)$ and $w(\theta)=z((l+4a)\theta)$ for $\theta\in[0,\f{l}{l+4a}]$. Since $D_\G$ is a quasidisk, we can find a map $f$ that is conformal in $D_\G$ and can be extended to a quasi-conformal mapping in the entire plane. Then on the boundary, $f$ maps $\mathbb{T}$ to the quasicircle $\pa D_\G$.

Let $\phi_f:=\phi\circ w^{-1}\circ f$. By the definition of $\mathcal{G}_v$, one has $\phi_f\in H^{1/2}(\mathbb{T})$. The rest of the proof is exactly the same as that of \cite[Lemma 8.1]{bishop}. The idea is to show $f^{-1}\circ w$ is a quasisymmetric map by definition, and then use the arguments by Beurling and Ahlfors \cite{ba} that $H^{1/2}$ is invariant under composition with a quasisymmetric homeomorphism of $\mathbb{T}$.

\noindent (2) By the $H^{3/2}$ characterization, we know $\int_0^l\int_0^l \left| \f{z'(t)-z'(s)}{t-s} \right|^2\,dtds<\infty$. Since $\G$ is chord-arc, $\f{|z(t)-z(s)|}{|s-t|}\in [\frac{1}{C},1]$ for some constant $C$. We have
\begin{align*}
\int_\G\int_\G \left( \f{|\nu(z)-\nu(w)|}{|z-w|} \right)^2\,|dz|\,|dw|&=\int_0^l\int_0^l \left( \f{|z'(t)-z'(s)|}{|z(t)-z(s)|} \right)^2\,dsdt\\
&\simeq \int_0^l\int_0^l \left| \f{z'(t)-z'(s)}{t-s} \right|^2\,dsdt<\infty.
\end{align*}
\end{proof}

\begin{rmk}
A direct consequence of the $H^{3/2}$ characterization is $\G$ has finite M\"{o}bius energy, i.e.
\beq\label{mobius}
M\ddot{o}b(\G)=\int_{\G}\int_\G\left( \f{1}{|z-w|^2}-\f{1}{l(z,w)^2} \right)\,dz\,dw<\infty.
\eeq
Here $l(z,w)$ is the length of $\G$ between $z$ and $w$.
For the proof, one can refer to the proof of \cite[Lemma 9.1]{bishop}. The brief idea is to show that $M\ddot{o}b\,(\G)\simeq \int_\G\int_\G \f{\int_{\gamma_{z,w}}\int_{\gamma_{z,w}}|\nu(x)-\nu(y)|^2|dx|\,|dy|}{|z-w|^4}|dz|\,|dw|$ and then change the order of integration. Furthermore, since $\G$ is chord-arc, it holds that
\beqo
\f{1}{|z-w|^2}-\f{1}{l(z,w)^2}=\f{(l(z,w)+|z-w|)(l(z,w)-|z-w|)}{|z-w|^2l(z,w)^2}\simeq \f{l(z,w)-|z-w|}{|z-w|^3}.
\eeqo
Then \eqref{mobius} implies that $\int_\G\int_\G\f{l(z,w)-|z-w|}{|z-w|^3}|dz|\,|dw|<\infty$.
\end{rmk}

\begin{rmk}
Other characterizations of Weil-Petersson curve in \cite{bishop} include approximation by polygons in a precise sense and the square integrability of $\beta$-numbers. The arguments also work for curves in $\mathcal{G}_v$ and we state them here without proof. Again we consider the arc-length parameterization $z(t): [0,l]\ri \G$. For each $n$, let $z_j^n=z\left(\f{jl}{2^n}\right)$ for $j=0,1,\dots,2^n$. Then it is obvious that $\{z_j^n\}$ divides $\G$ into $2^n$ intervals with equal length. Let $\G_n$ be the curve that consists of all the line segments $z_j^nz_{j+1}^n$ for $j=0,\dots,2^n-1$. Then one has
\beqo
\sum\limits_{n=1}^\infty 2^n[l(\G)-l(\G_n)]<\infty.
\eeqo
Recall the definition of Peter Jone's $\beta$-number: given a curve $\G$, $x\in \BR^2$ and $t>0$,
\beqo
\beta_\G(x,t):=\inf\limits_L\sup\limits_{z\in B(x,t)\cap \G}\f{\mathrm{dist}(z,L)}{t},
\eeqo
where the infimum is over all lines hitting $B(x,t)$. Then for $\G\in\mathcal{G}_v$, it satisfies
\beqo
\int_\G\int_0^\infty\beta_\G^2(x,t)\f{dt\,dx}{t^2}<\infty.
\eeqo
\end{rmk}

\section{Existence of minimizers}

The primary goal of this section is to establish the existence for Problem P. Before we state the theorem, we need to clarify some basic settings. Throughout this section we assume the volume $v=1$. We will consider $\G\in\mathcal{G}_1$ such that $E_{\G}\leq M$ for some constant $M>0$ and $\G$ will only intersects with $x-$axis at two endpoints. As a consequence, $\G$ has all the geometric properties that we have shown in Section 2 (Lemma \ref{reversedtriangleineq}, Lemma \ref{extension}, Proposition \ref{chordarc}, Proposition \ref{vanishingchordarc} and Proposition \ref{weil-peterson}).

Also we need to discuss different notions of boundary since we will perform integration by parts in $\Om_\G$. In geometric measure theory, there are three different kinds of boundary for a set $E$ of finite perimeter: topological boundary $\pa E$, measure-theoretical boundary $\pa^e E$ and reduced boundary $\pa^*E$. We refer to \cite{federer} for detailed definitions of these notions. It is well-known that $\pa^*E\subset \pa^e E\subset \pa E$. For our domain $\Om_\G$, it is obvious that $\pa\Om_\G=\G\cup\{(x,0),-a\leq x\leq a\}$. By Lemma \ref{reversedtriangleineq}, one can easily verify that for any $z\in \G\backslash \{(-a,0), (a,0)\}$, we have
\beqo
\liminf\limits_{r\ri 0} \f{|\Om_\G\cap B(z,r)|}{\pi r^2}>0,\quad \limsup\limits_{r\ri 0} \f{|\Om_\G\cap B(z,r)|}{\pi r^2}<1
\eeqo
This means $z\in\pa^e \Om_\G$, and therefore $\pa\Om_\G\backslash\pa^e\Om_\G\subset\{(-a,0), (a,0)\}$. As for the relation between measure-theoretical boundary and reduced boundary, a well-known result by Federer says that $\mathcal{H}^1(\pa^e E\backslash \pa^*E)=0$. So in the proof below, we don't distinguish these different notions of boundary when we write boundary integral.

\begin{theorem}\label{existence}
There exists a $\G\in \mathcal{G}_1$ minimizing the functional $E(\G)$ defined by \eqref{energy}.
\end{theorem}

\begin{proof}
Let $\{\G_i\}_{i=1}^{\infty}$ be a minimizing sequence in $\mathcal{G}_1$.
\beqo
\lim\limits_{i\ri\infty} E(\G_i)=M_0:=\inf\limits_{\G\in\mathcal{G}_1}E(\G).
\eeqo
Let $\Om_i,\T_i$ denote the corresponding $\Om_{\G_i}, \T_{\G_i}$ respectively. For each $i$, we set $(\pm a_i,0)$ as the two endpoints of $\G_i$. Because $l(\G_i)\leq M$ for every $i$, we have $a_i< \f{M}{2}$. Also by Lemma \ref{reversedtriangleineq} and the fact that $|\Om_i|=1$ we have $a_i\geq \f{2}{C}$ where $C$ is the constant in \eqref{reversetri}. Now we summarize
all the properties (independent of $i$) we need for $\{(\G_i,\Om_i, \T_i, a_i)\}$ before taking a limit.
\begin{enumerate}
\item[(a)] $\f{2}{C}\leq a_i\leq \f{M}{2}$, $l(\G_i)\leq M$.
\item[(b)] $\overline{\Om}_i\subset B(0,2M)$, $|\Om_i|=1$ and $\pa\Om_i=\G_i\cup\{(x,0):-a_i\leq x\leq a_i\}$.
\item[(c)] $\G_i$ can be parameterized by $(x_i(t),y_i(t))$ such that
\begin{align*}
x_i(0)=-a_i, \; x_i(l(\G_i))=a_i,\; y_i(0)=y_i&(l(\G_i))=0,\\
x'_i(t)\geq 0,\; y_i(t)\geq 0,\; |x_i'(t)|^2+|y_i'(t)|^2&=1,\, a.e.\\
 \sqrt{|x_i(t+s)-x_i(t)|^2+|y_i(t+s)-y_i(t)|^2}\geq& \f{s}{C_3} \text{ for }t+s<l(\G_i)
\end{align*}
where $C_3$ is the constant in Proposition \ref{chordarc}.
\item[(d)] $\T_i$ can be extended to a $H^1$ function on $B(0,2M)$ such that $\|\T_i\|_{H^1(B(0,2M))}\leq C_5$ for some universal constant $C_5$.
\end{enumerate}

Then we claim that there is a subsequence, still denoted by $\{(\G_i,\Om_i, \T_i, a_i)\}$ that converges in the following sense:
\begin{enumerate}
\item $\chi_{\Om_i}\ri \chi_{\Om}$ weakly in $BV(B(0,2M))$ and strongly in $L^1(B(0,2M))$, for some $\Om$ which is a set of finite perimeter in $B(0,2M)$ with volume $1$.
\item $\T_i\ri \T$ weakly in $H^1(B(0,2M))$ and strongly in $L^2(B(0,2M))$ for some $\T\in H^1(B(0,1))$.
\item $l(\G_i)\ri l$, $a_i\ri a$ for some constant $l>0,a>0$.
\item $\G_i\ri \G$ in Hausdorff distance for some chord-arc curve $\G$. $\G$ can be parameterized by $x(t),y(t)$ such that
    \begin{align*}
    &x(0)=-a, \; x(l)=a,\; y(0)=y(l)=0,\\
    &x'(t)\geq 0,\; y(t)\geq 0,\; |x'(t)|^2+|y'(t)|^2\leq 1,\, a.e.\\
    &\sqrt{|x(t+s)-x(t)|^2+|y(t+s)-y(t)|^2}\geq \f{s}{C_3} \text{ for }t+s<l
    \end{align*}
\item $\pa^e \Om\subset \G\cup \{(x,0): -a\leq x\leq a\}$, $\nu\cdot (\cos\T,\sin\T)=0$ a.e. on $\pa^*\Om$.
\end{enumerate}
\begin{proof}[Proof of the convergence claim]
(1), (2), (3) are straightforward to check. (4) is a direct consequence of Arzela-Ascoli lemma and Property (c) that we list before, we omitted the detail of the proof. So we only prove (5). First we show $\pa^e\Om\subset \G\cup\{(x,0): -a\leq x\leq a\}$. Assume there exists a point $z\in \pa^e\Om$ such that $z\not\in \G\cup\{(x,0): -a\leq x\leq a\}$. Then by convergence properties (3) and (4), there exists a $r_0>0$ and $n\in\mathbb{N}$ such that
\beqo
B(z,r_0)\cap \pa\Om_i=\varnothing, \; \forall i\geq n.
\eeqo
For any  $i\geq n$, we have
\beqo
\f{|\Om_i\cap B(z,r)|}{|B(z,r)|}=0 \text{ or }1, \text{ for any } r\leq r_0
\eeqo
However, by convergence property (1), we have
\beqo
\f{|\Om\cap B(z,r)|}{|B(z,r)|}=\lim\limits_{i\ri\infty} \f{|\Om_i\cap B(z,r)|}{|B(z,r)|}=0 \text{ or }1, \text{ for any }r\leq r_0,
\eeqo
which contradicts with our assumption $z\in \pa^e{\Om}$. Therefore we have proved $\pa^e\Om\subset \G\cup\{(x,0): -a\leq x\leq a\}$.

Now we set
\beqo
\Om_{in}:= \text{ the domain enclosed by }\G \text{ and x-axis},\quad \Om_{out}:=\mathbb{R}^2\backslash\left( \G\cup\{(x,0): -a\leq x\leq a\}\cup \Om_{in} \right)
\eeqo
By similar density argument one can show $\Om_{in}\subset \Om^{0}$ and $\Om_{out}\subset \Om^{1}$, where $\Om^{t}$ is defined as $\{z: \lim\limits_{r\ri 0}\f{|\Om\cap B(z,r)|}{\pi r^2}=t\}$. After a modification of a measure zero set, we can simply identify $\Om$ as $\Om_{in}$. Now we are left to show the second part of (5), which says the tangential anchoring boundary condition still holds for the limit domain.  Let $\phi$ be an arbitrary $C^\infty$ function in $\mathbb{R}^2$, we define
\beqo
n_i:=(\cos\T_i,\sin\T_i), \quad n:=(\cos\T,\sin\T),\quad \nu_i=\text{normal vector on } \pa\Om_i.
\eeqo
Note that here all $n_i$ and $n$ are defined on the larger domain $B(0,2M)$. We first deduce that
\beq\label{convdiv}
\lim\limits_{i\ri\infty} \int_{\Om_i}\di(\phi n_i)\, dx=\int_{\Om }\di(\phi n)\,dx.
\eeq
In fact,
\begin{align*}
&|\int_{\Om_i}\di(\phi n_i)\, dx-\int_{\Om }\di(\phi n)\,dx|\\
\leq & |\int_{\Om} \di(\phi n-\phi n_i)\,dx|+|\int_{\Om\Delta\Om_i} |\di(\phi n_i)|\,dx|
\end{align*}
As $i\ri\infty $, the first term goes to zero because $n_i$ converges to $n$ weakly in $H^1$; the second term goes to zero since $\Om_i$ converges to $\Om$ in $L^1$ and $n_i$ are uniformly bounded in $H^1$. Also we have that the following Gauss-Green formula holds
\beq\label{gaussgreen}
\int_{\Om_i} \di(\phi n_i)\,dx=\int_{\pa^*\Om_i} \phi n_i\cdot\nu_i\,d\mathcal{H}^1,\quad \int_{\Om} \di(\phi n)\,dx=\int_{\pa^*\Om} \phi n\cdot\nu\,d\mathcal{H}^1
\eeq
We want to point out that \eqref{gaussgreen} is not trivial here since $\pa\Om_i$ is not in $C^1$. However in our problem, it is valid because all $\Om_i$ and $\Om$ are Sobolev extension domains and one can define the trace of $H^1$ function on the reduced boundary. We refer to \cite[Proposition 3.4.4]{qli} or \cite[Theorem 3.84]{ambrosio} for more details. We may now combine \eqref{gaussgreen} with \eqref{convdiv} and get that
\beqo
0=\lim\limits_{i\ri\infty}\int_{\pa\Om_i}\phi n_i\cdot\nu_i\,d\mathcal{H}^1=\int_{\pa\Om}\phi n\cdot\nu\,d\mathcal{H}^1.
\eeqo
Thus the tangential anchoring boundary condition is proved for $\Om$ and $n$.
\end{proof}

On boundary $\pa\Om_i$ or $\pa \Om$, we define the tangent vector $\tau_i$, or correspondingly $\tau$, by rotating the normal vector $\nu_i $ or $ \nu$ by $\f{\pi}{2}$ clockwise. One can check that
\beqo
n_i=\tau_i \text{ on }\G_i\backslash \{(x,0): -a_i\leq x\leq a_i\},\quad  n_i=-\tau_i \text{ on } \{(x,0): -a_i\leq x\leq a_i\}\backslash \G_i,\quad \forall i\in\mathbb{N}
\eeqo
Using similar arguments in the proof of tangential anchoring condition above again, we can show that
\beq\label{tangential}
n=\tau \text{ on }\G\backslash \{(x,0): -a\leq x\leq a\},\quad  n=-\tau \text{ on } \{(x,0): -a\leq x\leq a\}\backslash \G,\quad a.e. \;x\in\pa\Om
\eeq
The idea is to carefully choose a cut-off function $\phi$ and calculate $\int_{\pa^*\Om} \phi n\cdot\tau\,d\mathcal{H}^1$ using Gauss-Green formula. We omit the details here. Note that \eqref{tangential} is equivalent to our original boundary condition \eqref{bdycondition}, therefore we have verified that $\G\in \mathcal{G}_1$ and $\Om,\, \T$ are just the corresponding $\Om_\G,\,\T_\G$.

Finally, by convergence result (1--5) and lower semi-continuity we conclude that
\beqo
\int_{\Om}|\na \T|^2+l(\G)\leq \liminf\limits_{i\ri\infty} E(\G_i)=M_0
\eeqo
And by Lemma \ref{mini}, $\G$ won't touch $x-$axis besides two endpoints. So $(\G,\Om,\T)$ is a minimizer of Problem P. The proof is complete.

\end{proof}

Next we want to study the behavior of $\G\in \mathcal{G}_1$ near $(-a,0)$ and $(a,0)$. The following lemma indicates that $\G$ and $x-$axis form approximately cusps near two ends. Note that here we don't assume $\G$ is a minimizer.
\begin{lemma}\label{cusp}
Let $\G\in \mathcal{G}_1$ satisfy $E(\G)\leq M$. $\G$ only intersects with $x-$axis at $z_1=(-a,0)$ and $z_2=(a,0)$. For any $k>0$, there exists a constant $r$ that depends on $k$ and $\G$ such that
\begin{align*}
&\text{If }z=(x,y)\in \G\cap B(z_1,r), \text{ then } \f{y}{x+a}\leq k,\\
&\text{If }z=(x,y)\in \G\cap B(z_2,r), \text{ then } \f{y}{a-x}\leq k.
\end{align*}
\end{lemma}
\begin{rmk}
This lemma implies that as $z\in \G$ approaches $z_1(\text{or }z_2)$, the angle between the ray $z-z_1(\text{or }z-z_2)$ and $x-$axis converges to $0$.
\end{rmk}
\begin{proof}[Proof of Lemma \ref{cusp}]
Without loss of generality, we only prove the lemma near $z_1=(-a,0)$. We argue by contradiction. Assume the Lemma is false, there would exist a constant $k>0$, a sequence of radiuses $\{r_i\}_{i=1}^\infty$ and a sequence of points $\{(x_i,y_i)\}_{i=1}^\infty\subset\G$ such that
\beqo
r_i\ri 0,\quad \sqrt{(x_i+a)^2+y_i^2}=r_i,\quad \f{y_i}{x_i+a}=k.
\eeqo
By Lemma \ref{extension} we can extend $\T$ from $\Om_\G$ to the whole $\mathbb{R}^2$ such that $\|\T\|_{H^1(\mathbb{R}^2)}\leq C$. For every $i$, we introduce the following rescaled functions:
\begin{align*}
&\G_i:=\{\f{z-z_1}{r_i}:z\in \G\cap B(z_1,r_i)\},\\
&\Om_{i}:=\{\f{z-z_1}{r_i}:z\in \Om_\G\cap B(z_1,r_i)\},\\
&\T_i(z):=\T(z_1+r_iz) \text{ for }z\in B(0,1)
\end{align*}

One can easily check the following properties hold
\begin{enumerate}
\item[(a)] $l(G_i)\leq C_3$ for the constant $C_3$ from Proposition \ref{chordarc}.
\item[(b)] $\Om_i\subset B(0,1)$, $|\Om_i|\geq \f{\arctan{k}}{2}-\f{k}{2(1+k^2)}=:C_6$
\item[(c)] $\G_i$ can be parameterized by $(x_i(t),y_i(t))$ such that
\begin{align*}
x_i(0)=0, \; x_i(l(\G_i))=\f{1}{\sqrt{1+k^2}},\; y_i(0)=0,\; & y(l(\G_i))=\f{k}{\sqrt{1+k^2}},\\
x'_i(t)\geq 0,\; y_i(t)\geq 0,\; |x_i'(t)|^2+|y_i'(t)|^2&=1,\, a.e.\\
 \sqrt{|x_i(t+s)-x_i(t)|^2+|y_i(t+s)-y_i(t)|^2}\geq& \f{s}{C_3} \text{ for }t+s<l(\G_i)
\end{align*}
\item[(d)] $\{\T_i\}_{i=1}^\infty$ is uniformly bounded in $H^1(B(0,1))$ and we have
\beqo
\lim\limits_{i\ri\infty}\int_{B(0,1)}|\na \T_i|^2\,dx=0
\eeqo
\end{enumerate}

Passing if necessary to a subsequence, we get
\begin{enumerate}
\item $\chi_{\Om_i}\ri \chi_{\Om}$ weakly in $BV(B(0,1))$ and strongly in $L^1(B(0,1))$, for some $\Om$ with volume lower bound $|\Om|\geq C_6$.
\item $\T_i\ri \T$ weakly in $H^1(B(0,1))$ and strongly in $L^2(B(0,1))$ for some $\T\in H^1(B(0,1))$.
\item $l(\G_i)\ri l$ for a constant $l>0$.
\item $\G_i\ri \G^*$ in the sense of Hausdorff distance for some chord-arc curve $\G^*$. $\G^*$ can be parameterized by $x(t),y(t)$ such that
    \begin{align*}
    &x(0)=0, \; x(l)=\f{1}{\sqrt{1+k^2}},\; y(0)=0,\; y(l)=\f{k}{\sqrt{1+k^2}},\\
    &x'(t)\geq 0,\; y(t)\geq 0,\; \f{1}{C_3}\leq |x'(t)|^2+|y'(t)|^2\leq 1,\, a.e.
    \end{align*}
\item $\pa^e \Om \subset \G^*\cup \{(x,0): 0\leq x\leq 1\}\cup \pa B(0,1)$, $\quad \nu\cdot (\cos\T,\sin\T)=0$ a.e. on $\pa^*\Om$.
\end{enumerate}
The proof of the above convergence property is the same as Theorem \ref{existence}. By lower semi-continuity and weak convergence of $\T_i$ in $H^1(B(0,1))$ we have
\beqo
\int_{B(0,1)}|\na \T|^2\,dx\leq \lim\limits_{i\ri\infty }\int_{B(0,1)}|\na \T_i|^2\,dx=0
\eeqo
Therefore $\T$ is a constant function, which contradicts with (5) since the normal vector of $\pa\Om$ obviously cannot be orthogonal to a constant vector by simple geometry.
\end{proof}

As for the regularity of $\G$ away from two endpoints, Proposition \ref{vanishingchordarc} and Proposition \ref{weil-peterson} tells that $\nu$ belongs to $\mathrm{VMO}$ and $H^{1/2}([0,l])$. Unfortunately this is the best regularity result we have now. Here we give the following natural open problems:

\begin{problem}{1}
Is $\G$ a $C^\infty$ curve, or at least $C^1$?
\end{problem}
\begin{problem}{2}
Can one write $\G$ as a curve of function $f(x)$, such that $|\f{df}{dx}|\leq C$ for some constant $C<\infty$?
\end{problem}

The difficulty in answering these questions is due to the strong non-local character of the tangential anchoring boundary condition. It prevents us from modifying $\G$ locally to obtain an energy competitor and then deduce decay of some energy quantities. Therefore some new ideas and methods are needed in order to utilize the minimality. We now assume the statement in Problem 2 is true, and we compute the Euler-Lagrange equation that $f$ should satisfy.

Let $\G=\{(x,f(x)):x\in[-a,a]\}$ such that
\beqo
f(-a)=f(a)=0, \quad f(x)>0\,\text{ for }x\in(-a,a),\quad \text{ and }\,\int_{-a}^a f(x)\,dx=1.
\eeqo

We write $\Om_\G$ as $\Om_f$. Then we consider the perturbation of $f(x)$ and domain $\Om_f$
\beqo
f_t(x)=f(x)+tg(x),\, \Om_{f_t}(x)=\{(x,y):x\in[-a,a],\, y\in[0,f(x)+tg(x)]\},
\eeqo
where $g\in C_0^1(-a,a)$.
We denote the domain variation by $\Phi(t,x)$ such that
\beqo
\Phi(t,\Om_f)=\Om_{f_t},\, \Phi(t,(x,y))=(x,y+\f{y}{f(x)}tg(x))
\eeqo
By this definition, we can see that $\Phi(t)$ satisfies that
\beqo
\Phi(0)=I,\, \Phi'(0):=\f{d}{dt}(\Phi(t)-I)=V(x,y)=(0,\f{y}{f(x)}g(x)).
\eeqo
$\T(t,z)$ solves the equation
\beq\label{eqn:Thetat}
\begin{cases}
-\D \T(t,z)=0&\text{ in }\Om_{f_t},\\
\T(t)=\arctan{(f'+tg')} &\text{ on }\pa\Om_{f_t}.
\end{cases}
\eeq

Here $z=(x,y)\in\BR^2$. The functional becomes
\beqo
F(f_t)=\int_{-a}^a \sqrt{1+|f_t'(x)|^2}\,dx+ \int_{-a}^a\int_0^{f_t(x)}|\na \Theta(t,z)|^2\,dxdy=:F_1(f_t)+F_2(f_t).
\eeqo
Suppose $f$ is smooth and $t\ri \T(t)$ has good differentiability properties(denote by $\T'$ its derivative at $0$). We can differentiate (\ref{eqn:Thetat}) inside $\Om_f$ and at the boundary, by differentiating the following identity:
\beqo
\text{For }z=(x,f(x)), \quad \Theta(t,\Phi(t,z))=\arctan{(f'+tg')}.
\eeqo
We obtain
\begin{align*}
&-\D \T'=0\text{ in }\Om_f\\
&\T'(x,f(x))+\na \T(x,f(x)) \cdot V=\f{g'(x)}{1+|f'(x)|^2}  \text{ on } \G\\
&\T'=0 \text{ on }\{(x,0):x\in[-a,a]\}.
\end{align*}
Therefore, $\T'$ is the harmonic function with a Dirichlet boundary condition(depending on $f,g$). We compute the derivative of $F(f_t)$ at $t=0$. For the first part, we easily obtain
\beqo
\f{d}{dt}F_1=\int_{-1}^1\f{f'g'}{\sqrt{1+|f'|^2}}dx.
\eeqo
For the second part, we have
\begin{align*}
\f{d}{dt}F_2&=\int_{\Om_f}\bigg\{2\na\T'\cdot\na \T +\mathrm{div}(|\na\T|^2 V)\bigg\}\,dxdy\\
&=\int_{\G} \bigg\{   2\f{\pa\T}{\pa \nu}(\f{g'}{1+|f'|^2}-\na\T\cdot V) +|\na\T|^2(V\cdot \nu)  \bigg\} \,d\mathcal{H}^1\\
&=\int_{-a}^a\bigg\{    2\f{\pa\T}{\pa \nu}\f{g'}{\sqrt{1+|f'(x)|^2}}-2(\na \T\cdot \nu)\cdot(\na \T\cdot V)\sqrt{1+|f'(x)|^2}+g(x)|\na \T|^2      \bigg\}\,dx
\end{align*}
where $\nu$ is the normal vector on $\G$. Here we have used the boundary condition of $\T'$ and integration by parts. Also we have used the following formula
\beqo
V(z)=(0,g(x)),\quad  \nu(z)=\f{(-f'(x),1)}{\sqrt{1+|f'(x)|^2}}\quad \text{ for }\,z=(x,f(x))\in \pa\Om_f.
\eeqo
Let $\f{d}{dt}F(f_t)=0$ and take into account the volume constraint, we obtain the following Euler-Lagrange equation
\beq\label{EL}
\lambda = -\f{d}{dx}\left(\f{f'}{\sqrt{1+|f'|^2}}\right)
-2\f{d}{dx}\left(   \f{\pa\T}{\pa \nu}\f{1}{\sqrt{1+|f'(x)|^2}} \right)-\f{\pa \T}{\pa \nu}\f{\pa\T}{\pa y}\sqrt{1+|f'|^2}+|\na \T|^2
\eeq
where $\lam$ is the Lagrange multiplier and the derivative of $\T$ is taking value at $(x,f(x))$. Note that \eqref{EL} is complicated and contains some highly non-local terms, such as the Dirichlet-to-Neumann map $\f{\pa\T}{\pa\nu}$. The first part of the equation is the minimal surface equation while the rest comes from the Dirichlet energy with tangential anchoring condition and domain variation. It will be very interesting to study the well-posedness of \eqref{EL} and we believe that the key of solving the regularity problem of $\G$ is to understand this equation.

\section{Large volume limit and small volume limit}

In this section we study the behavior of the minimizer as the volume $v$ tends to be extremely large or small. A naive idea is to analysis the functional \eqref{energy} from a scaling point of view. The curve length term is of dimension one while the Dirichlet energy term is of dimension zero. Therefore, when the volume is very large, the first term will be the dominating term and the minimizer is expected to be close to a semicircle (minimizes length of graph under fixed volume constraint). On the other hand, when the volume is very small, the domain is energy preferable to be very thin to avoid large elastic energy. We will present more rigorous analysis in the rest of this section.

\subsection{Large volume limit}
Since we are only interested in the shape of $\G$, we will modify Problem P and restrict $a=1$. First we make the following notations:
\beqo
\mathcal{G}_v^a:=\{\G\in\mathcal{G}_v:\, \G \text{ only intersects with $x$-axis at }(a,0),\,(-a,0) \},\quad \mathcal{G}^a:=\bigcup\limits_{v>0} \mathcal{G}_v^a.
\eeqo
Then we can write Problem P as
\beqo
\min\limits_{a>0}\min\limits_{\G\in \mathcal{G}_v^a} \bigg\{ \int_{\Om_\G} |\na \Theta|^2\,dxdy+l(\G) \bigg\}.
\eeqo
Let $\bar{x}=\f{x}{a}$, $\bar{y}=\f{y}{a}$, $\bar{\Theta}(\bar{x},\bar{y})=\Theta(a\bar{x},a\bar{y})$, the minimization problem becomes
\beqo
\min\limits_{a>0}\min\limits_{\G\in \mathcal{G}_{v/a^2}^1} \bigg\{ \int_{\Om_\G} |\na \Theta|^2\,dxdy+a\cdot l(\G) \bigg\}.
\eeqo
Setting $\tilde{a}=\f{a}{\sqrt{v}}$ leads to
\beq\label{minimization1}
\min\limits_{\tilde{a}>0}\min\limits_{\G\in \mathcal{G}_{1/\tilde{a}^2}^1} \bigg\{ \int_{\Om_\G} |\na \Theta|^2\,dxdy+\tilde{a}\sqrt{v}\cdot l(\G) \bigg\}.
\eeq
This is equivalent to
\beq\label{minimization2}
\min\limits_{\G\in \mathcal{G}^1} \bigg\{ \int_{\Om_\G} |\na \Theta|^2\,dxdy+\sqrt{v}\f{ l(\G)}{\sqrt{|\Om_\G|}} \bigg\}.
\eeq
When $v\gg 1$, we consider the following functional for $\G\in \mathcal{G}^1$:
\beqo
E_v(\G)=\f{1}{\sqrt{v}}\int_{\Om_\G} |\na \Theta|^2\,dxdy+\f{ l(\G)}{\sqrt{|\Om_\G|}}
\eeqo
We denote by $\G_v$ the minimizer of functional $E_v(\G)$. As $v\ri +\infty$, one expects that $\G_v$ will ``converge" in some proper sense to $\G^*:=\{(x,\sqrt{1-x^2}): x\in[-1,1]\}$, which is well known to minimize the following functional
\beqo
F(\G)=\f{ l(\G)}{\sqrt{|\Om_\G|}}\; \text{ for }\G\in\mathcal{G}^1.
\eeqo
We have the following lemma:

\begin{lemma}
$\lim\limits_{v\ri\infty }E_v(\G_v)= \sqrt{2\pi}=F(\G^*)$
\end{lemma}
\begin{proof}
We borrow the idea of "adding two cusps" from \cite{gnsv}. We modify $\G^*$ near $x=-1$ and $x=1$ by adding two cusps. For $\e<<1$,  we define a function $\bar{f}^\e$ as follows
\beqo
\bar{f}^\e(x)=\begin{cases}
\sqrt{1-x^2},& |x|\leq \sqrt{1-\e^2},\\
\f{\e}{1-\e}-\sqrt{(\f{\e}{1-\e})^2-(x+\sqrt{\f{1+\e}{1-\e}})^2}, & x\in (-\sqrt{\f{1+\e}{1-\e}},-\sqrt{1-\e^2}),\\
\f{\e}{1-\e}-\sqrt{(\f{\e}{1-\e})^2-(-x+\sqrt{\f{1+\e}{1-\e}})^2}, & x\in (\sqrt{1-\e^2},\sqrt{\f{1+\e}{1-\e}}).
\end{cases}
\eeqo
Note that here we change the graph near two endpoints of $\G^*$ into two circular arcs to make sure the derivative of $\bar{f}^\e$ vanishes near two end points (See Figure \ref{fig3}).

\begin{figure}[H]
\centering
\begin{tikzpicture}[scale=3.9]
\clip (-1.7,-0.3) rectangle (1.7,1.2);
\draw[thick] (-1.4,0)--(1.4,0);
\draw[thick] ({cos(15)},{sin(15)}) arc [start angle=15, end angle=165, radius=1];
\draw[dashed] ({cos(15)},{sin(15)}) arc [start angle=15, end angle=0, radius=1];
\draw[dashed] ({-cos(15)},{sin(15)}) arc [start angle=165, end angle=180, radius=1];
\draw[thick] ({cos(15)},{sin(15)}) arc [start angle=195, end angle=270, radius={sin(15)/(1-sin(15))}];
\draw[thick] ({-cos(15)},{sin(15)}) arc [start angle=-15, end angle=-90, radius={sin(15)/(1-sin(15))}];
\filldraw[black] ({cos(15)},{sin(15)}) circle[radius=0.4pt] node[left] {$(\sqrt{1-\e^2},\e)$};
\filldraw[black] ({-cos(15)},{sin(15)}) circle[radius=0.4pt] node[right] {$(-\sqrt{1-\e^2},\e)$};
\filldraw[black] ({sqrt((1+sin(15))/(1-sin(15)))},0) circle[radius=0.4pt] node[below] {$(\sqrt{\frac{1+\e}{1-\e}},\e)$};
\draw[dashed] (0,0)--({sqrt((1+sin(15))/(1-sin(15)))},{sin(15)/(1-sin(15))});
\draw[dashed] ({sqrt((1+sin(15))/(1-sin(15)))},{sin(15)/(1-sin(15))})-- node[right] {$r=\frac{\e}{1-\e}$}({sqrt((1+sin(15))/(1-sin(15)))},0);
\end{tikzpicture}
\caption{The graph of function $\bar{f}^\e$}
\label{fig3}
\end{figure}
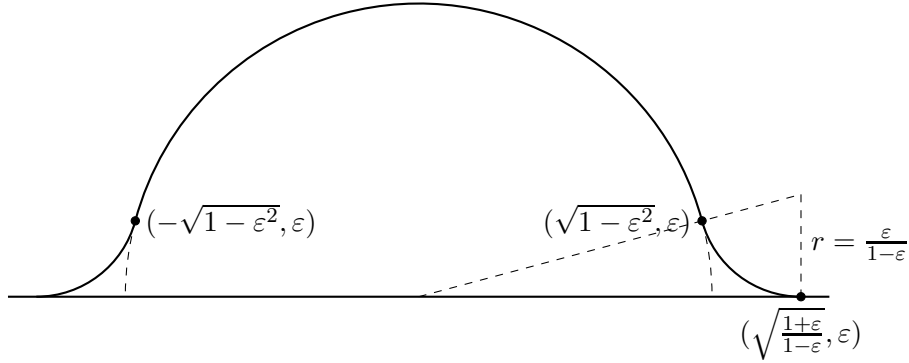

Then we set for any $\e\in (0,\f14)$,
\begin{align*}
f^\e(x)&=\bar{f}^\e(\sqrt{\f{1+\e}{1-\e}}x),\\
\G^\e&=\{(x,f^\e{x}): x\in[-1,1]\},\\
\T^\e(x,y)&=\arctan{\f{d}{dx}f^\e(x)}\quad \text{for }(x,y)\in \G^\e\\
\T^\e(x,y)&=\f{\arctan{\f{d}{dx}f^\e(x)}}{f^\e(x)}\cdot y\quad \text{for }(x,y)\in \Om_{\G^\e}
\end{align*}
It is straightforward to check that $\G^\e$ satisfies the following property:
\begin{enumerate}
  \item $\G^\e\in\mathcal{G}^1$,
  \item $\lim\limits_{\e\ri 0} F(\G^\e)=\sqrt{2\pi}$,
  \item $\int_{\Om_{\G^\e}}|\na \T^\e|^2\,dydx=O(\f{1}{\e})$.
\end{enumerate}
For $v\gg 1$, we set $\e=v^{-\f14}$. Then $\f{1}{\sqrt{v}}\int_{\Om_{\G^\e}}|\na \T^\e|^2\,dydx=O(v^{-\f14})$. This implies
\beqo
\lim\limits_{v\ri\infty} E_v(\G_v)\leq \lim\limits_{v\ri\infty} E_v(\G^{\e})=\sqrt{2\pi}\leq \lim\limits_{v\ri\infty}E_v(\G_v).
\eeqo
\end{proof}

\begin{rmk}
If $\lim\limits_{v\ri \infty}F(\G_v)=\sqrt{2\pi}$, then
\beqo
\lim\limits_{v\ri \infty }|\Om_{\G_v}\Delta\Om_{\G^*}|=0,\quad \lim\limits_{v\ri \infty }\mathrm{d}_{\mathcal{H}}(\G_v ,\G^*)=0,
\eeqo
where $\mathrm{d}_{\mathcal{H}}$ is the Hausdorff distance. This is an easy consequence of the stability of isoperimetric inequality (see \cite{fmp}).
\end{rmk}

\subsection{Small volume limit}

First we prove the following lemma which provides a rough estimate for the Dirichlet energy when the volume of droplet is sufficiently small.
\begin{lemma}\label{smallenergy}
Take $\e\ll 1$, there exist constants $c$ and $C$ which are independent of $\e$, such that for any $\G\in\mathcal{G}_{\e}^1$, it holds that
\beqo
c\,\e\leq \int_{\Om_\G} |\na \T|^2 \,dxdy \leq C\,\e.
\eeqo
\end{lemma}
\begin{proof}
Given $\G\in\mathcal{G}_{\e}^1$, $\T$ is the corresponding angle function.
We first estimate the lower bound of energy. We set
\beqo
\tilde{\G}=\{(x,y): (x,\e y)\in\G\},\quad \tilde{\T}(x,y)=\arctan \left(\f{\tan{\T(x,\e y)}}{\e}\right) \text{ for }(x,y)\in\Om_{\tilde{\G}}.
\eeqo
Then we can check that $|\Om_{\tilde{\G}}|=1$ and $\T$ satisfies the boundary condition \eqref{bdycondition} corresponding to $\tilde{\G}$. Thus there exists a constant $c$ such that
\beq\label{smallvol1}
\int_{\Om_{\tilde{\G}}}|\pa_y\tilde{\T}|^2\,dxdy\geq c
\eeq
Otherwise one can use the similar argument in Section 3 to get a contradiction. On the other hand, by definition of $\tilde{\G}$ and $\tilde{\T}$ we have
\begin{align}\label{smallvol2}
\int_{\Om_{\tilde{\G}}}|\pa_y\tilde{\T}|^2\,dxdy&= \int_{\Om_{\G}}\f{|\pa_y\T|^2}{\e\cdot \left(1+\left|\f{\tan{\T(x,y)}}{\e}\right|^2\right)^2\cdot |\cos\T(x,y)|^4}\,dxdy\\
\nonumber &\leq \f{1}{\e}\int_{\Om_\G} |\pa_y\T|^2\,dxdy.
\end{align}
Therefore
\beqo
\int_{\Om_\G}|\na \T|^2\,dxdy\geq c\e,
\eeqo
by \eqref{smallvol1} and \eqref{smallvol2}.  Meanwhile, we can construct a $\G\in\mathcal{G}_{\e}^1$ such that $\int_{\Om_\G} |\na \T|^2 \,dxdy =\leq C\e$ for some larger constant $C$. Set
\begin{align*}
\G=&\{(x, \f{\e}{2}(\cos x+1)):x\in[-1,1]\}\\
\T(x,y)=&\f{\arctan{f'(x)}}{f(x)}y,\quad\text{for }x\in(-1,1),\,y\in[0,f(x)].
\end{align*}
We can directly verify that $\G\in \mathcal{G}_\e^1$ and $\int_{\Om_\G} |\na \T|^2 \,dxdy \leq C\e$ for some constant $C$ independent of $\e$. This proves Lemma \ref{smallenergy}.
\end{proof}

\begin{rmk}
Now we consider the minimization problem \eqref{minimization1} with $v=\e^2\ll 1$. We can determine the appropriate order of $\tilde{a}$. Assume $\tilde{a}\sim O(\e^{-\al})$ for some $\al\in\mathbb{R}$. Then the second term (surface energy term) is of order $\e^{1-\al}$. For the Dirichlet energy term since $\G\in\mathcal{G}_{\e^{2\al}}^1$, by Lemma \ref{smallenergy} we know it is of order $\e^{2\al}$. Matching these two terms gives $\al=\f13$. According to the deduction of \eqref{minimization1} we know that if we don't fix two endpoints of $\G$, then the energy-minimizing droplet with volume $\e^2$ will be a elongated drop with length of the order $\e^{\f23}$ and the total energy is of order $\e^{\f23}$.
\end{rmk}

Next we study the asymptotic shape of the rescaled droplet. For such purpose, we add some extra regularity assumption on the curve $\G$. Consider a subset of $\mathcal{G}^1$, denoted by $\tgg$, which consists of all the curves in $\mathcal{G}^1$ that are graphs of $H^{2}_0$ functions,
\beqo
\tgg:=\{\G\in\mathcal{G}^1,\; \G=\{(x,f(x))\},
\eeqo
where $f$ satisfies
\beq\label{conditionf}
f\in H_0^{2}([-1,1]),\; f'(\pm1)=0,\;f(x)>0 \text{ on }(-1,1).
\eeq
Given $\e\ll 1$, we define a transformation operator $\mathcal{T}_\e$, which compresses $\G\in\tgg$ in the vertical direction:
\beqo
\mathcal{T}_{\e}(\G)=\{(x,\e^{\f23}f(x)):\;x\in[-1,1]\}, \quad \G=\{(x,f(x))\}.
\eeqo
Now after taking $v=\e^2$ in \eqref{minimization2} and multiplying $\e^{-\f23}$, we obtain the functional
\beq
\begin{aligned}\label{Eepsilon}
E_\e(f)&=E_\e(\G)\\
&=\e^{-\f23}\int_{\Om_{\mathcal{T}_{\e}(\G)}}|\na\T_{\mathcal{T}_{\e}(\G)}|^2\,dxdy+\e^{\f13}\f{l(\mathcal{T}_{\e}(\G))}{\sqrt{\Om_{\mathcal{T}_{\e}(\G)}}}\\
&=\e^{-\f23}\int_{\Om_{\mte{\G}}}|\na\T_{\mte(\G)}|^2\,dxdy+\f{l(\mte(\G))}{\sqrt{\Om_{\G}}}\\
&=\e^{-\f23}\int_{-1}^1 \int_{0}^{\e^{\f23}f(x)}\{|\pa_x \T_{\mte(\G)}|^2+|\pa_y \T_{\mte(\G)}|^2\}\,dydx+\f{\int_{-1}^1 \sqrt{1+\e^{\f43}|f'(x)|^2}\,dx}{\sqrt{\int_{-1}^1 f(x)\,dx}}.
\end{aligned}
\eeq
Then for a sequence of positive numbers $\e\ri 0$, we consider the sequence of functionals on $H_0^{2}([-1,1])$
\beq
E_\e(f):=\begin{cases}
E_\e(f)\text{ defined in \eqref{Eepsilon}}, &\text{if } \G=\{x,f(x)\}\in\mathcal{G}^1,\\
+\infty &\text{otherwise}.
\end{cases}
\eeq
And we also define the candidate functional $E_0(f)$ for $\G$-convergence,
\beqo
E_0(f):=\int_{-1}^1 \f{|f'{x}|^2}{f(x)}\,dx+\f{2}{\sqrt{\int_{-1}^1 f(x)\,dx}},\quad f\in H_0^{2}([-1,1]).
\eeqo
We have the following result:
\begin{proposition}\label{gammacvg}
As $\e\ri 0$, the sequence $\{E_\e\}$ $\G$-converges to $E_0$ in the $H^2$ topology.
\end{proposition}
\begin{proof}
First we prove the lower semi--continuity condition, i.e. for any $g\in C_0^1[-1,1]$ and for any sequence $\{g_\e\}$ in $C_0^1[-1,1]$,
\beq\label{lsc}
g_\e\ri g \text{ in }H^2[-1,1] \text{ implies }\liminf\limits_{\e\ri 0}E_\e(g_\e)\geq E_0(g).
\eeq
The case $\liminf\limits_{\e\ri 0}E_\e(g_\e)=+\infty$ is trivial. We therefore assume that $\liminf\limits_{\e\ri 0}E_\e(g_\e)=C<+\infty$. And by the $C^1$ convergence of $g_\e$, we also suppose that $|g_\e'(x)|\leq c$ for some constant $c$ holds for any $\e>0$ and $x\in[-1,1]$. Now we examine the first term of $E_\e(g_\e)$ more closely
\begin{align*}
&\e^{-\f23}\int_{-1}^1\int_0^{\e^{\f23}g_\e(x)}\{|\pa_x\T_{\mte(\G)}|^2+|\pa_y\T_{\mte(\G)}|^2\}\,dydx\\
>&\e^{-\f23}\int_{-1}^1\int_0^{\e^{\f23}g_\e(x)}\{|\pa_y\T_{\mte(\G)}|^2\}\,dydx\\
\geq& \e^{-\f23}\int_{-1}^1\left\{\f{|\T_{\mte(\G)}(x,\e^{\f23}g_\e(x))|^2}{\e^{\f23}g_\e(x)}\right\}\,dx\\
=&\e^{-\f23}\int_{-1}^1\left\{\f{|\arctan{(\e^{\f23}g_\e'(x))}|^2}{\e^{\f23}g_\e(x)}\right\}\,dx
\end{align*}
Since $|g_\e'(x)|\leq c$, we have that for any $\sigma>0$, there exists $\e_\s>0$ such that for any $\e<\e_\s$, $|\arctan{(\e^{\f23}g_\e'(x))}|\geq (1-\s)|\e^{\f23}g_\e'(x)|$. And therefore we have
\begin{align*}
&\e^{-\f23}\int_{-1}^1\left\{\f{|\arctan{(\e^{\f23}g_\e'(x))}|^2}{\e^{\f23}g_\e(x)}\right\}\,dx\\ \geq &\e^{-\f23}(1-\s)^2\int_{-1}^1\left| \f{\e^{\f43}|g_\e'(x)|^2}{\e^{\f23}g_\e(x)} \right|\,dx=(1-\s)^2\int_{-1}^1 \f{|g_\e'(x)|^2}{g_\e(x)}\,dx,\quad \text{ when }\e<\e_\s
\end{align*}
We obtain
\begin{align*}
&\liminf\limits_{\e\ri 0}E_\e(g_\e)\\
=&\liminf\limits_{\e\ri 0}\left\{\e^{-\f23}\int_{-1}^1\int_0^{\e^{\f23}g_\e(x)}\{|\pa_x\T_{\mte(\G)}|^2+|\pa_y\T_{\mte(\G)}|^2\}\,dydx+\f{\int_{-1}^1 \sqrt{1+\e^{\f43}|g_\e'(x)|^2}\,dx}{\sqrt{\int_{-1}^1 g_\e(x)\,dx}}\right\}\\
\geq& \liminf\limits_{\e\ri 0}\left\{\e^{\f23}\int_{-1}^1 \f{|g_\e'(x)|^2}{g_\e(x)}\,dx+\f{2}{\sqrt{\int_{-1}^1 g_\e(x)\,dx}}\right\}\geq E_0(g)
\end{align*}
Here in the last step we used the $C^1$ convergence of $g_\e$ and Fatou's lemma. This gives the proof of the lower semi-continuity \eqref{lsc}.

The second part of proving Gamma-convergence is to find a recovery sequence for each $f$ satisfying \eqref{conditionf}. We can simply take $f_\e=f$ for any $\e>0$. By the same argument in the proof of lower semi-continuity, we have
\beqo
\lim\limits_{\e\ri 0} E_\e(f)\geq E_0(f)
\eeqo
On the other hand, take $\T_\e(x,y)=\f{y}{\e^{\f23}f(x)}\arctan{(\e^{\f23}f'(x))}$ for $(x,y)$ satisfying $-1\leq x\leq 1,\; 0\leq y\leq e^{\f23}f(x)$. It holds that
\begin{align*}
&\e^{-\f23}\int_{-1}^1\int_0^{\e^{\f23}f(x)}|\pa_x\T_{\e}|^2\,dydx\\
=&\int_{-1}^1 \f{\e^{\f43}}{3}f(x)^3\left|\f{f''}{f(1+\e^{\f43}|f'|^2)}-\f{f'\arctan{(\e^{\f23}f')}}{\e^{\f23}f^2} \right|^2\,dx\\
\sim &O(\e^{\f43}).
\end{align*}
\begin{align*}
&\e^{-\f23}\int_{-1}^1\int_0^{\e^{\f23}f(x)}|\pa_y\T_{\e}|^2\,dydx\\
=&\int_{-1}^1 \f{|\arctan{(\e^{\f23}f')}|^2}{\e^{\f43}f}\,dx\sim  O(1).
\end{align*}
After comparing the above two identities, we conclude that
\begin{align*}
&\e^{-\f23}\int_{-1}^1\int_0^{\e^{\f23}f(x)}|\na \T_\e|^2\,dydx+\f{\int_{-1}^1 \sqrt{1+\e^{\f43}|f'(x)|^2}\,dx}{\sqrt{\int_{-1}^1 f(x)\,dx}}\\
=&(1+o(1))\int_{-1}^1 \f{|f'|^2}{f}\,dx+\f{2}{\sqrt{\int_{-1}^1 f(x)\,dx}}+o(1)=E_0(f)+o(1)
\end{align*}
Therefore we obtain $\lim\limits_{\e\ri 0}E_\e(f)=E_0(f)$ for any $f\in H_0^{2}([-1,1])$. The proof is complete.

\end{proof}

Proposition \ref{gammacvg} inspires us to study the following minimization problem
\beq\label{minimization4}
\min\limits_{g\in H_0^{2}([-1,1])} \bigg\{ \f{2}{\sqrt{\int_{-1}^1 g(x)\,dx}}+ \int_{-1}^1 \f{|g'(x)|^2}{g(x)}\,dx \bigg\}
\eeq
Let $g=h^2$, the problem becomes
\beq\label{minimization3}
\min\limits_{h^2\in H^2_0[-1,1]} \bigg\{ \f{2}{\sqrt{\int_{-1}^1 h(x)^2\,dx}}+ 4\int_{-1}^1 |h'(x)|^2\,dx \bigg\}
\eeq
The Euler Lagrange equation is
\beqo
h''(x)=-\f{h(x)}{4(\int_{-1}^1 h^2 \,dx)^{\f32}}, \quad h\in H^2_0[-1,1].
\eeqo
This ODE can be solved explicitly,
\beqo
h(x)=\pi^{-\f23}\cos{\f{\pi}{2}x}
\eeqo
and therefore
\beqo
g(x)=\pi^{-\f43}\big(\f{1+\cos{\pi x}}{2}\big)
\eeqo
is the minimizer for the minimization problem \eqref{minimization4}. Using the above $\G$-convergence result, we conclude that when the volume $v=\e^2<<1$, the approximated profile of $\G$ is $\Large\{\left(x,\e^{\f23}\pi^{-\f43}(\f{1+\cos{\pi x}}{2})\right):x\in[-C,C]\Large\}$, where $C\sim O(\e^{\f23})$ is a coefficient that ensures the volume constraint.

\end{document}